\documentclass[11pt]{amsart}
\usepackage{geometry}                
\usepackage{graphicx,mathrsfs}
\usepackage{amssymb}
\usepackage{epstopdf, color}
\usepackage{amsaddr}

\input xy
\xyoption{all}
\CompileMatrices

\title{Essential dimension of generic symbols in characteristic $p$}
\author[K. McKinnie]{Kelly McKinnie}
\address{Department of Mathematical Sciences, University of Montana, Missoula, MT 59812}
\email{kelly.mckinnie@umontana.edu}

\newtheorem{theorem}{Theorem}[section]
\newtheorem{lemma}[theorem]{Lemma}
\newtheorem{proposition}[theorem]{Proposition}
\newtheorem{corollary}[theorem]{Corollary}
\newtheorem*{main}{Main Result}
\newtheorem*{maincor}{Corollary}

\theoremstyle{remark}
\newtheorem{example}[theorem]{Example}
\newtheorem{remark}[theorem]{Remark}

\renewcommand{\H}{\mathrm{H}}
\newcommand{\trdeg}{\mathrm{tr.deg}}
\newcommand{\res}[1]{\mathrm{res}_{#1}}
\newcommand{\gen}{\mathrm{gen}}
\newcommand{\ed}{\mathrm{ed}}

\newcommand{\ur}{{p,\mathrm{ur}}}
\newcommand{\Z}{\mathbb Z}
\newcommand{\br}{\mathrm{Br}}
\newcommand{\alg}{\mathrm{Alg}}
\newcommand{\kln}{k_{\ell,n}}
\newcommand{\okln}{\overline k_{\ell,n}}
\newcommand{\kpln}{k'_{\ell-1,n}}
\newcommand{\vyln}{v_{{\ell,n}}}
\newcommand{\dec}{\mathrm{Dec}}
\newcommand{\pgl}{\mathrm{PGL}}

\numberwithin{equation}{section}
\begin{document}

\begin{abstract}
In this article the $p$-essential dimension of generic symbols over fields of characteristic $p$ is studied. In particular, the $p$-essential dimension of the length $\ell$ generic $p$-symbol of degree $n+1$ is bounded below by $n+\ell$ when the base field is algebraically closed of characteristic $p$. The proof uses new techniques for working with residues in Milne-Kato $p$-cohomology and builds on work of Babic and Chernousov in the Witt group in characteristic 2.  Two corollaries on $p$-symbol algebras (i.e, degree 2 symbols) result from this work. The generic $p$-symbol algebra of length $\ell$ is shown to have $p$-essential dimension equal to $\ell+1$ as a $p$-torsion Brauer class. The second is a lower bound of $\ell+1$ on the $p$-essential dimension of the functor $\alg_{p^\ell,p}$. Roughly speaking this says that you will need at least $\ell+1$ independent parameters to be able to specify any given algebra of degree $p^{\ell}$ and exponent $p$ over a field of characteristic $p$ and improves on the previously established lower bound of 3. 
\end{abstract}

\maketitle

\section{Introduction}\label{s1}

The essential dimension of an algebraic object is informally defined as the number of algebraically independent parameters you need to define the object. In this paper we will consider the essential dimension of objects and functors relating to central simple algebras and higher symbols in Milne-Kato cohomology, focusing on the bad characteristic case. 
Since its introduction in \cite{MR1457337}, most of the upper and lower bounds on the essential dimension of central simple algebras have required that the degree of the algebra be relatively prime to the characteristic of the base field $k$. Two excellent surveys on essential dimension,  \cite{MR3055773} and \cite{Reichstein-ICM}, contain many of these results, algebraic and functorial definitions of essential dimension, $p$-essential dimension and much more. 

When the characteristic of $k$ divides the degree of the central simple algebra, the so-called ``bad characteristic case'', upper and lower bounds on the essential dimension have been more sparse.  To illustrate that these cases are fundamentally different we can look at {\it generic symbol algebras} in both cases.  First, let us fix some notation. Let $m$ and $n$ be positive integers with $m|n$, let $p>0$ be prime and fix a field $k$. Define functors $\alg_n$, $\alg_{n,m}$, $\dec_{p^\ell}$, ${}_p\br:\textbf{Fields}/k \to \textbf{sets}$ by 
\begin{eqnarray*}
\alg_n(K)& = & \{\textrm{isom. classes of central simple $K$-algebras of degree $n$}\}\\
\alg_{n,m}(K)& = & \{\textrm{subset of $\alg_n(K)$ whose elements have exponent dividing $m$}\}\\
\dec_{p^\ell}(K)& = & \{\textrm{isom. classes of tensor products of $\ell$ degree p symbols over $K$}\}\\
{}_p\br(K) & = & \{\textrm{$p$-torsion Brauer classes over $K$}\}
\end{eqnarray*}
for any field extension $K/k$. $\H^1(K,\pgl_n)$, the set of isomorphism classes of $\pgl_n$-torsors over $\mathrm{Spec}(K)$, has a bijective correspondence with $\alg_n(K)$, the set of isomorphism classes of central simple algebras of degree $n$ over $K$. In particular, using the standard notation in \cite{Reichstein-ICM}, $\ed_k(\pgl_n) = \ed_k(\alg_n)$.

In good characteristic, we can follow Example 2.8 of \cite{Reichstein-ICM}. Let $p$ be a prime and let $k$ be a field containing a primitive $p$-th root of unity, $\omega$. Let $x_i$ and $y_i$ be algebraically independent indeterminates over $k$ and set $K=k(x_i,y_i)_{i=1}^\ell$. Consider the length $\ell$ generic symbol $K$-algebra $A_\ell = \otimes_{i=1}^\ell(x_i,y_i)_\omega$. $A_\ell$ is a central simple $K$-algebra of degree $p^\ell$ and exponent $p$. The essential dimension of $A_\ell$ as both an element of $\alg_{p^\ell}(K)$ and of ${}_p\br(K)$ is $2\ell$ as one might suspect \cite[2.6]{Reichstein-ICM}, giving a lower bound $\ed_k(\pgl_{p^\ell}) = \ed_k(\alg_{p^\ell})\geq 2\ell$.  

On the other hand, if the characteristic of $k$ is $p$ we can consider the analogous algebra $D_\ell = \otimes_{i=1}^\ell [x_i,y_i)$ over $K=k(x_i,y_i)_{i=1}^\ell$. In \cite[3.2]{baek} Baek shows $\ed_k(\dec_{p^\ell})\leq \ell+1$, assuming $k$ contains the field with $p^\ell$ elements.  In particular, the essential dimension of $D_\ell$ as an element of $\dec_{p^\ell}(K)$ (and hence also as an element of $\alg_{p^\ell,p}(K)$, $\alg_{p^\ell}(K)$ and ${}_p\br(K)$) is at most $\ell+1$.  We call $D_\ell$ the {\it length $\ell$ generic $p$-symbol} and motivation for this paper comes from finding its $p$-essential dimension as an element in ${}_p\br(K)$ (Corollary \ref{c34}).  

As noted in \cite[\S10.1]{MR1321649}, for a field $F$ of characteristic $p$ the Milne-Kato $p$-cohomology group $\H^{n+1}_p(F)$ is defined to be analogous to the Galois cohomology group $\H^{n+1}(F, \mu_p^{\otimes n})$ when characteristic $F\ne p$. This analogy is made precise in \cite{MR689394}. When $n=1$ these groups each realize the $p$-torsion in $\br(k(x_i,y_i)_{i=1}^\ell)$ and thus contain the classes of the generic symbol algebras $A_\ell$ (when $\mu_p\subset k$) and $D_\ell$ (when char$(k)=p$). In both types of cohomology one can generalize the notion of generic symbols to higher degrees. The main result of this paper finds a lower bound on the $p$-essential dimension of the length $\ell$ generic $p$-symbols in $\H^{n+1}_p(k_{\ell,n})$ when $k$ is algebraically closed of characteristic $p$. More specifically, fix $k$ algebraically closed of characteristic $p$ and for $x_i$ and $y_{i,j}$  algebraically independent indeterminates over $k$ set 
\begin{equation}
k_{\ell,n} = k(x_i,y_{i,j})_{1\leq i \leq \ell}^{1\leq j \leq n}
\label{e8080}
\end{equation}
so that $\trdeg_k(\kln) = \ell(n+1)$. The {\it length $\ell$ generic $p$-symbol of degree $n$}, $\gen_k(n+1,\ell,p)$, is defined as the class 
\begin{equation}
\gen_k(n+1,\ell,p) = \sum_{i=1}^\ell x_i\frac{dy_{i,1}}{y_{i,1}}\wedge\cdots\wedge\frac{dy_{i,n}}{y_{i,n}} \in \H^{n+1}_p(k_{\ell,n})
\label{e1111}
\end{equation}
(see \S\ref{s3}, \cite[\S10.1]{MR1321649}, \cite{I}, \cite{MR689394} for the definition of $\H^{n+1}_p$). Let $\ed_k(\gen_k(n+1,\ell,p))$ denote essential dimension and let $\ed(\gen_k(n+1,\ell,p);p)$ denote $p$-essential dimension of $\gen_k(n+1,\ell,p)$ as an element of $\H^{n+1}_p(k_{\ell,n})$. A lower bound on the $p$-essential dimension of $\gen_k(n+1,\ell,p)$ is our main result;

\begin{main}[Theorem~\ref{t1}]
Let $k$ be an algebraically closed field of characteristic $p$. For $\ell,n\geq 1$, 
\[\ed_k(\gen_k(n+1,\ell,p))\geq \ed_k(\gen_k(n+1,\ell,p);p)\geq \ell+n.\]
\end{main}

\noindent In degree 2, $\gen_k(2,\ell,p) = D_\ell$ and the theorem tells us that $\ed_k(D_\ell)\geq \ell+1$ as an element of $\H^2_p(k_{\ell,1}) = {}_p\br(k_{\ell,1})$. Combining this with the upper bound from \cite[3.2]{baek} we have

\begin{maincor}[Corollary \ref{c34}]
Let $k$ be an algebraically closed field of characteristic $p$. For $\ell \geq 1$, 
\[\ed_k(D_\ell) =\ed_k(D_\ell;p) =  \ell+1\]
where the essential dimensions are taken with respect to $D_\ell \in {}_p\br(k_{\ell,1})$.
\end{maincor}

When char$(k)=p$ the best known bounds on $\ed_k(\alg_{p^\ell,p^r})$ and $\ed_k(\alg_{p^\ell,p^r};p)$ are as follows.
\begin{itemize}
\item[-] In \cite[2.2]{baek} Baek gives a lower bound 
\[\ed_k(\alg_{p^\ell,p^r};p)\geq 3\]
on the $p$-essential dimension when $1\leq r<\ell$. This result holds regardless of the characteristic of $k$.
\item[-] By  \cite[Ex. 1.1]{MR2520892} for any field $k$ and any integers $1\leq m\leq n$ with $m|n$, $\ed_k(\alg_{n,m}) = \ed_k(\mathrm{GL}_n/\mu_m)$ and $\ed_k(\alg_{n,m};p) = \ed_k(\mathrm{GL}_n/\mu_m;p)$. Using this, recent work by Garibaldi and Guralnick, \cite[6.7]{GG}, gives an upper bound
\[\ed_k(\alg_{p^\ell,p^r})\leq p^{2\ell}-3p^\ell+p^{\ell-r}\]  for $p^\ell \geq 4$. This bound is also independent of the characteristic of $k$.
\end{itemize}

As a corollary to Corollary \ref{c34}, we improve on the lower bound of 3 when char$(k)=p$ and $r=1$.

\begin{maincor}[Corollary~\ref{main}]
Let $k$ be an algebraically closed field of characteristic $p$. 
\[\ed_k(\alg_{p^\ell,p};p)\geq \ell+1.\]
\end{maincor}

\begin{proof} The algebra $D_\ell$ satisfies $\ed_k(D_\ell;p)\geq \ell+1$ as an element of $\alg_{p^\ell,p}(k_{\ell,1})$ since it satisfies the same inequality as an element of ${}_p\br(k_{\ell,1})$ by Corollary \ref{c34}.\end{proof}

\begin{remark}Milne-Kato cohomology groups have also been used to study essential dimension in bad characteristic in \cite{baek2016}.  Baek finds non-trivial cohomological invariants into Milne-Kato cohomology groups to prove the lower bound $\ed(\pgl_4)\geq 4$ over a field of characteristic 2.
\end{remark}

\section{Generic symbols with char$(k) \ne p$, methods and outline.} 
A  discussion of a lower bound for the essential dimension of generic symbols with char$(k)\ne p$ provides a proper overview of the char$(k)=p$ arguments and illustrates a major difficulty we will encounter when char$(k) = p$. Let $k$ be an algebraically closed field with char$(k)\ne p$. The Galois symbol gives the analogue of the generic $p$-symbols defined above.  That is, let $\kln$ be defined as in (\ref{e8080}) and let $h^{n+1}_{\kln,p}:K^M_{n+1}(\kln) \to \H^{n+1}(k_{\ell,n},\mu_p^{\otimes (n+1)})$ be the Galois symbol map as defined in \cite[4.6.4]{GS}. Define
\[\gen_k(n+1,\ell,p) = h^{n+1}_{k_{\ell,n},p}\left(\sum_{i=1}^\ell\{x_i,y_{i,1},\ldots,y_{i,n}\}\right).\]
In the case $n=1$ if we fix a primitive $p$-th root of unity $\omega$, and with it an isomorphism $\H^2(k_{\ell,1},\mu_p^{\otimes 2})\cong {}_p\br(k_{\ell,1})$, then $\gen_k(2,\ell,p) = [A_\ell]$ from above.
Using the methods of this paper we can find the same lower bound on the essential dimension of these generic symbols as elements of $\H^{n+1}(k_{\ell,n},\mu_p^{\otimes (n+1)})$ as in Theorem \ref{t1};

\begin{proposition}Let $\gen_k(n+1,\ell,p) \in \H^{n+1}(k_{\ell,n},\mu_p^{\otimes n+1})$ be defined as above. Then
 \begin{equation*}
 \ed_k(\gen(n+1,\ell,p))\geq\ed_k(\gen(n+1,\ell,p);p)\geq \ell+n.
 \end{equation*}
 \label{e887}
 \end{proposition}
\begin{proof}  The proof is by induction on the length, $\ell$. If the $p$-essential dimension is less than $n+1$ then there exists a prime to $p$ field extension $K/k_{1,n}$, a subfield $k\subset E \subset K$ with $\trdeg_k(E)=n$ and $g \in \H^{n+1}(E,\mu_p^{\otimes n+1})$ so that $\res{K}(\gen_k(n+1,1,p)) = \res{K}(g)$. Any such field $E$ satisfies $\H^{n+1}(E,\mu_p^{\otimes(n+1)})=0$ \cite[6.5.14]{MR2392026}. Thus, to finish the case $\ell=1$, it is enough to show the following lemma.
\begin{lemma} Let $K/k_{1,n}(z_1,\ldots,z_r)$ be a prime to $p$ extension with $z_i$ algebraically independent over $k_{1,n}$ and $e$ an integer with $p\nmid e$. Then $e\,\res{K}(\gen_k(n+1,1,p)) \ne 0$. \label{newlemma}
\end{lemma}
\begin{proof}[Proof of lemma] As mentioned above, when $n=1$, the class of $\gen_k(2,1,p) = [(x_1,y_1)_\omega] = [A_1]$ is non-trivial in ${}_p\br(k(x_1,y_1))$. Moreover, for any integer $e$ with $p\nmid e$, $e\,\res{K}[A_1] \in{}_p\br(K)$ with $K$ as in the statement of the lemma is non-trivial \cite[3.6 \& 3.15b]{MR1692654}. Fix $n>1$, $K$ as in the statement of the lemma and assume $e\,\res{K'}(\gen_k(n_0+1,1,p)) \ne 0$ for all $n_0<n$ and all $K'$ as in the statement of the lemma. 

Let $(K,v)$ be an extension of $(k_{1,n}(z_1,\ldots,z_r),v_{1,n})$, where $v_{{1,n}}$ is the $y_{1,n}$-adic valuation on $k_{1,n}(z_1,\ldots,z_r)$, such that $e(v/v_{1,n})$ and $f(v/v_{1,n})$ are each prime to $p$. Set $\overline K$ and $\overline k_{\ell,n}$ to be the residue fields, respectively. Let $\xi = \{x_1,y_{1,1},\ldots,y_{1,n}\}\in K^M_{n+1}(k_{1,n})$ so that $\gen_k(n+1,1,p) = h_{k_{1,n},p}^{n+1}(\xi)$. If $e\,\res{K}(\gen(n+1,1,p)) = 0$ then the residue $\partial^{n+1}_v(e\,\res{K}(h^{n+1}(\xi)))=0$, \cite[6.8.5]{GS}.  The Galois symbol, residue map \cite[6.8.5]{GS} and tame symbol $\partial^M:K^M_n(K) \to K^M_{n-1}(\overline K)$ \cite[7.1]{GS} act as follows with respect to restriction of scalars:
\begin{eqnarray}
0 &=& \partial^{n+1}_v(e\,\res{K}(h^{n+1}(\xi)))\label{e454}\\
&=&\partial^{n+1}_v(e\,h^{n+1}(\res{K}(\xi)))\nonumber\\
&=&e\,h^n(\partial^M(\res{K}(\xi))) \qquad \textrm{(\cite[7.5.1]{GS})}\nonumber\\
&=&e\,h^n(e(v/v_{1,n})\,\res{\overline K}(\partial^M(\xi)))\qquad \textrm{(\cite[7.1.6(2)]{GS})}\nonumber\\
&=&e\,e(v/v_{1,n})\,h^n(\res{\overline K}(\{x_1,y_{1,1},\ldots,y_{1,n-1}\})\nonumber\\
&=& e\,e(v/v_{1,n})\,\res{\overline K}(h^n(\{x_1,y_{1,1},\ldots,y_{1,n-1}\}))\nonumber\\
&=& e\,e(v/v_{1,n})\,\res{\overline K}(\gen_k(n,1,p))\label{e44}
\end{eqnarray}
Since $e\,e(v/v_{1,n})$ is prime to $p$, $\overline k_{1,n} \cong k_{1,n-1}$ and $\overline K/ \overline k_{1,n-1}(z_1,\ldots,z_r)$ is a prime to $p$ extension, by the induction hypothesis (\ref{e44}) is non-trivial, a contradiction to (\ref{e454}). 
\end{proof}

Fix $\ell>1$ and assume the theorem holds for $\gen_k(n+1,\ell_0,p)$ for all $\ell_0<\ell$. Let $K/\kln$ be a prime to $p$ field extension, $k\subset E\subset K$ a field of transcendence degree $\ell+n-1$ over $k$ and $g \in \H^{n+1}(E,\mu_p^{\otimes(n+1)})$ such that $\res{K}(g) = \res{K}(\gen_k(n+1,\ell,p))$. As above, let $(K,v)$ be an extension of $(\kln,v_{{\ell,n}})$ with $e(v/v_{\ell,n})$ and $f(v/v_{\ell,n})$ prime to $p$. Let $w=v|_E$. The valuation $w$ cannot be trivial on $E$ because if it were the residue $\partial^{n+1}_v(\res{K}(\gen_k(n+1,\ell,p)))\in\H^n_p(\overline K,\mu_p^{\otimes n})$ would be zero. However, a computation similar to (\ref{e454})-(\ref{e44}) shows that 
\begin{equation}\partial^{n+1}_v(\res{K}(\gen_k(n+1,\ell,p))) = e(v/v_{\ell,n})\,\res{\overline K}(h^{n}_{\kln,p}(\{x_\ell,y_{\ell,1},\ldots,y_{\ell,n-1}\})).\label{e8989}\end{equation}
After renumbering, $h^{n}_{\kln,p}(\{x_\ell,y_{\ell,1},\ldots,y_{\ell,n-1}\}) = \gen_k(n,1,p)$ and $\overline K$ is a prime to $p$ extension of $\okln$ which is a purely transcendental extension of $k_{1,n-1}$. Therefore, by Lemma \ref{newlemma}, the right hand side of (\ref{e8989}) is non-zero, a contradiction to the triviality of $w$. 

Two crucial things happen when $w$ is nontrivial: first $\trdeg(\overline E)= \ell+n-2$ and second, the specialization $s^{n+1}_w$ and residue $\partial^{n+1}_w$ of $g$ are defined. This is a major point, we {\it can} take the specialization and residue (called the first and second residue in case char$(k)=p$) of $g$ because these maps are defined on all of $\H^{n+1}(E,\mu_p^{\otimes(n+1)})$. In the characteristic $p$ case, the first and second residues are only defined on the  0-th piece of the filtration of Izhboldin (see \S \ref{s3} or \cite[\S2]{I}) and though $\gen_k(n+1,\ell,p)$ is easily shown to be contained within the 0-th piece, there is no easy reason that $g$, the element it descends to, is contained within the 0-th piece. 

Back to the char$(k)\ne p$ case. Let $\pi$ be a uniformizer for $(K,v)$ and $\tau = u\pi^e$ a uniformizer for $(E,w)$ with unit $u \in K$. Under  extension of scalars $E\subset K$ the specialization and residue maps satisfy
\begin{eqnarray}
\partial^{n+1}_v(\res{K/E}(g))&=&e\,\partial^{n+1}_w(g)\label{e998}\\
s^{n+1}_v(\res{K/E}(g))&=&s^{n+1}_w(g)+ \partial^{n+1}_w(g)\cup (u)\label{e990}
\end{eqnarray}
If $p|e$ then the right hand side of (\ref{e998}) is zero whereas, since $\res{K/E}(g) = \res{K}(\gen_k(n+1,\ell,p))$, the left hand side is non-zero (\ref{e8989}). 
Therefore $p\nmid e$.  Since $p\nmid e$, $\partial^{n+1}_w(g) = e^{-1}h^{n+1}_{\kln,p}(\{x_\ell,y_{\ell,1},\ldots,y_{\ell,n-1}\})$ is split in the algebraic closure $k'=k(x_\ell,y_{\ell,1},\ldots,y_{\ell,n-1})^{alg}$. Replace the algebraically closed field $k$ with the algebraically closed field $k'$ and take composite fields: $K'=\overline K\cdot \kpln \subset \overline k_{\ell,n}^{alg}$ and $E' = \overline E \cdot k'$. Our field diagram looks like:
\[
\xymatrix@1@C=.1in@R=.2in{
     &K'\ar@{-}[dr]\\
E'\ar@{-}[ur]&\kpln\ar@{-}[u]&\overline K\\
                    & k'\ar@{-}[ul]\ar@{-}[u] &\okln\ar@{-}[u]\ar@{-}[ul]&\overline E\ar@{-}[ul]\\
                    &&k\ar@{-}[ur]\ar@{-}[u]\ar@{-}[ul]
}
\] 
Note two things here; since $\trdeg_k(\overline E) = n+\ell-2$, $\trdeg_{k'}(E')\leq n+\ell-2$ and since $p\nmid [\overline K:\okln]$, $p\nmid [K':k'_{\ell-1,n}]$. Let $y_{\ell,n}=u'\pi^{e'}$ with $u'$ a unit in $K$ and $p \nmid e'=e(v/v_{\ell,n})$. Since  
\[s_v^{n+1}(\res{K}(\gen_k(n+1,\ell,p))) = \res{\overline K}(\gen_k(n+1,\ell-1,p))+\partial^{n+1}_{v_{\ell,n}}(\gen_k(n+1,\ell,p))\cup (\bar u')\]
and $\res{K'}(\partial^{n+1}_{v_{\ell,n}}(\gen_k(n+1,\ell,p))) = 0$, using (\ref{e990}) we have
\begin{eqnarray*}
\res{K'}(\gen_k(n+1,\ell-1,p))&=&\res{K'}(s_v^{n+1}(\res{K}(\gen_k(n+1,\ell,p)))\\
&=&\res{K'}(s_v^{n+1}(\res{K}(g)))\\
&=&\res{K'}(s_w^{n+1}(g)+\partial^{n+1}_w(g)\cup(\bar u))\\
&=&\res{K'}(s_w^{n+1}(g))
\end{eqnarray*}
Since $\res{K'}(\gen_k(n+1,\ell-1,p)) = \res{K'}(\gen_{k'}(n+1,\ell-1,p))$ and $\res{K'}(s_w^{n+1}(g)) = \res{K'}(\res{E'}(s_w^{n+1}(g)))$, this shows that after the prime to $p$ extension $K'/k'_{\ell,-1,}$, the generic $p$-symbol, $\gen_{k'}(n+1,\ell-1,p)$ of length $\ell-1$ descends to the field $E'$ with $\trdeg_{k'}(\overline E')\leq n+\ell-2$, contradicting the induction hypothesis.
\end{proof}
\begin{remark}These arguments are reproduced (with more detail) in the proof of Theorem \ref{t1} in the bad characteristic case. Moreover, the lower bound in Proposition \ref{e887} is not optimal at least in the case $n=1$ and $\ell>1$ by the remark in section 1 \cite[2.6]{Reichstein-ICM} and probably more generally.
\end{remark}

{\bf Methods and outline.} As mentioned in the proof of Proposition \ref{e887}, much of the difficulty of the proof of Theorem \ref{t1} lies in the need to reduce to the case when $g$ is in the 0-th piece of Izhboldin's filtration on $p$-cohomology. This is done by building on the work done by Babic and Chernousov in \cite{bc}. In their paper so-called canonical monomial quadratic forms
\[t_1[1,x]\oplus t_2[1,x]\oplus\cdots\oplus t_n[1,x]\oplus \mathbb H\oplus\cdots\oplus \mathbb H\]
over $k(t_1,\ldots,t_n,x)$ with char$(k)=2$ are shown to be incompressible. Here $[a,b]$ is the quadratic form $ax^2+xy+by^2$ in characteristic 2. The incompressibility of these forms gives them a lower bound on $\ed_k(\mathbf O(V,g))$ where $g$ is any non-degenerate quadratic form on a vector space $V$ over $k$. In the present paper the techniques from \cite[\S7-\S12]{bc} are adapted to both the Milne-Kato $p$-cohomology and the generic forms in (\ref{e1111}) to get the lower bound in Theorem \ref{t1}. 

In \S\ref{s2} differential bases over fields of characteristic $p$ are reviewed and  \cite[11.1]{bc}  is generalized in Proposition \ref{p1} to arbitrary prime characteristic. Proposition \ref{p1} serves in this paper, as 11.1 does in \cite{bc}, as a keystone of the proofs on essential dimension that follow.  In \cite[\S 8]{bc} Babic and Chernousov use a presentation of  quadratic forms in the Witt group over a field of Laurent series by Arason. In this paper we instead work with Izhboldin's filtration on $\H^{n+1}_p(F)$ for $F$ a Laurent series field \cite[\S2]{I}. When $p=2$, $\H^{n+1}_2(F)$ is isomorphic to a homogeneous component of the graded Witt group (\cite{MR662605}), but for $p>2$ there is no such connection. The appropriate adaptations for $p$-cohomology are done in \S\ref{s3}. In Lemma \ref{l1} we show there is a ``unique decomposition'' of $p$-cohomology classes, similar to the unique decomposition in \cite[8.2]{bc}. Work is done in Proposition \ref{p2} to understand how one manipulates a $p$-cohomology class into its ``unique decomposition''.  The proof of the main theorem and corollaries are in section \ref{s4}.

\section{Differential bases in characteristic $p$}\label{s2}
Throughout this section let $k$ be a perfect field of characteristic $p$ and $K$ a field containing $k$ with $\trdeg_{k}K = r>0$. Let $v$ be a geometric valuation on $K$ of rank 1 (so that $\trdeg_k(\overline K) = r-1$ where $\overline K$ is the residue field of $K$ \cite{MR1914006}). Note that with this set up $K/K^p$ is a defectless extension, that is, $[v(K):v(K^p)] = p$ and $[\overline K: \overline K^p] = p^{r-1}$ so that $p^r=[K:K^p] = [v(K):v(K^p)]\cdot [\overline K:\overline K^p]$. Set $\pi$ as a uniformizer for $v$ and $R\subset K$ the corresponding valuation ring. 

As in \cite[\S9]{bc} we will say that  a differential basis $\{a_1,\ldots, a_r\}$ for $K/k$ {\it comes from} $\overline K$ if there exists an $i_0$ with $a_{i_0}$ a uniformizer for $K$, $a_j \in R^\times$ for $j \ne i_0$ and $\{\overline a_j\,|\,j \ne i_0\}$ is a differential basis for $\overline K/k$. (See \cite[16.5]{E} for the equivalence of differential bases and $p$-bases.) 

\begin{example}
Let $k$ be a field of characteristic $p$. Take $k(t_1,\ldots, t_r)$ to be the rational function field in $r$ variables over $k$, $v$ the $t_r$-adic valuation, $\pi = t_r$ and $R=k(t_1,\ldots, t_{r-1})[t_r]$. In this case $\{t_1,\ldots, t_r\}$ forms a differential basis of $k(t_1,\ldots,t_r)/k$ coming from $\overline{k(t_1,\ldots,t_{r})} \cong k(t_1,\ldots,t_{r-1})$. Let $K/k(t_1,\ldots,t_r)$ be a prime to $p$ field extension. There exists an extension of the valuation $v_{t_r}$ on $k(t_1,\ldots,t_r)$ to a discrete valuation $v$ on $K$ with residue degree $f(v/v_{t_r})$ and ramification index $e(v/v_{t_r})$ both prime to $p$ \cite[16.6.3]{MR2752311}. Since $K/k(t_1,\ldots,t_r)$ is a finite prime to $p$ extension, it is separable algebraic and thus the differential basis $\{t_1,\ldots,t_r\}$ of $k(t_1,\ldots,t_r)$ is also a differential basis of $K/k$ (see \cite[8.6]{MR2058673}). Let $\tau \in K$ be a uniformizer for the extended valuation $v$. The set $\{t_1,\ldots,t_{r-1},\tau\}$ is a differential basis of $K/k$ which comes from $\overline K$ since $\{t_1,\ldots,t_{r-1}\}$ is a differential basis for $k(t_1,\ldots,t_{r-1})$ and hence also for the prime to $p$ extension $\overline K$ (see \cite[section 9]{bc}). Note also that $K$ has transcendence degree $r$ over $k$ and the valuation $v$ on $K$ is geometric of rank 1. 
\label{ex1}
\end{example}

\begin{proposition}[Generalization of {\cite[11.1]{bc}} to characteristic $p$] Let $k$ be a perfect field of characteristic $p$ and $K$ a field containing $k$ with $\trdeg_{k}K = r>0$. Let $v$ be a geometric valuation on $K$ of rank 1. Let $E\subset K$ be a subfield containing $k$ with $\trdeg_k(E)=s<r=\trdeg_k(K)$. Then there exists a differential basis $\{a_1,\ldots, a_r\}$ of $K/k$ coming from $\overline K$ such that $E\subset K^p(a_1,\ldots, a_t)$ with $t \leq s <r$.
\label{p1}
\end{proposition}

\begin{proof}
(Follows \cite[11.1]{bc}) Since $k$ is perfect we can fix a $p$-basis $\{c_1,\ldots, c_s\}$ of $E/k$ so that $E=E^p(c_1,\ldots, c_s)$. Set $L=K^p(c_1,\ldots,c_s)$ so that $E\subset L$. After reordering if necessary let $c_1,\ldots, c_t$ be a minimal system of generators for $L$ over $K^p$. Let $F_0 = K^p\subset F_1\subset \cdots \subset F_{r}=K$ be any chain of degree $p$ extensions which are built using the $c_i$:
\[F_0=K^p\subset F_1 = F_0(c_1)\subset F_2 = F_1(c_2)\subset \cdots \subset F_{t} = F_{t-1}(c_t) = L\]
Consider the corresponding chain of residue fields
\[\overline K^p = \overline F_0\subset \overline F_1\subset \overline F_2\subset \cdots \subset \overline F_t = \overline L\subset \overline F_{t+1}\subset \cdots \subset \overline F_N = \overline K\]
Since $v$ is geometric, $[\overline K:\overline K^p] = p^{r-1}$, showing that exactly one of these $r$ residue extensions is trivial and the rest have degree $p$. For each nontrivial extension choose $\overline a_i \in \overline F_i \backslash \overline F_{i-1}$ and any lift $a_i \in F_i\backslash F_{i-1}$. This is the part of the differential basis that comes ``from $\overline K$''. Let $\overline F_{i_0} = \overline F_{i_0-1}$ be the collapsed part of the residue fields. We need to find a uniformizer $a_{i_0} \in F_{i_0}$ for $K$ which completes the differential basis. Since $K/K^p$ is defectless, the subextension $F_{i_0-1}\subset F_{i_0}$ is also defectless, so that 
\[p=[F_{i_0}:F_{i_0-1}] = [v(F_{i_0}):v(F_{i_0-1})]\cdot [\overline F_{i_0}:\overline F_{i_0-1}]= [v(F_{i_0}):v(F_{i_0-1})].\]
In particular we can find $\gamma \in F_{i_0}$ with $p\nmid v(\gamma)$. Take $\alpha, \beta \in \Z$ with $\alpha p+\beta v(\gamma) = 1$ and set $a_{i_0} = \pi^{p\alpha}\gamma^\beta$. Note that $v(a_{i_0}) = 1$ and $F_{i_0-1}(a_{i_0}) = F_{i_0-1}(\gamma) = F_{i_0}$. Hence $\{a_1,\ldots,a_r\}$ forms a differential basis of $K/k$ coming from $\overline K$ with $E\subset L=K^p(a_1,\ldots,a_t)$ and $t\leq s <r$.
\end{proof}

\begin{remark} Note that if $v|_E$ has ramification index a multiple of $p$, then each $c_i$ has value a multiple of $p$. In particular, in the proof of Proposition \ref{p1}, the collapse $\overline F_{i_0} = \overline F_{i_0-1}$ must happen with $i_0>t$. Therefore in this case, for $i\leq t$, $a_i \in R^\times$.
\label{r3}
\end{remark}

As a result of Proposition \ref{p1} we will be interested in subfields of the form $L = K^p(a_1,\ldots, a_s)\subset K = K^p(a_1,\ldots, a_r)$ with $\{a_i\}_{i=1}^r$ a differential basis of $K$. The following Lemma will be used in the proof of Theorem \ref{t22}. Set $\Lambda_{s} = \Z_p^{s}$ and use multi-index notation $e=(e_1,\ldots, e_{s}) \in \Lambda_{s}$ to write $a^e := a_1^{e_1}\cdots a_{s}^{e_{s}}$. In this way the set $\{a^e\,|\,e \in \Lambda_{s}\}$ forms a $K^p$-basis for $L$.

\begin{lemma} Let $L = K^p(a_1,\ldots, a_s)\subset K = K^p(a_1,\ldots, a_r)$ with $\{a_i\}_{i=1}^r$ a differential basis of $K$. If $s<n$ then the restriction of scalars map $\Omega^{n}_L \to \Omega^{n}_K$ is the zero map.
\label{l999}
\end{lemma}

\begin{proof}
Let $bdc_{1}\wedge\cdots\wedge dc_{n}\in \Omega^{n}_L$ be a $n$-form and write 
\[b= \sum_{  e \in \Lambda_{s}} \beta_{e}^pa^{e} \hspace{.25in} \textrm{and} \hspace{.25in}
c_{i} = \sum_{  e' \in \Lambda_{s}} \gamma_{ie'}^pa^{ e'}\]
with $\beta_{e}, \gamma_{ie'} \in K$. Extend scalars from $L$ to $K$  and use the fact that $\beta_e^p$ and $\gamma_{ie'}^p$ are now $p$th powers to expand $bdc_{1}\wedge\cdots\wedge dc_{n}$ into a sum of elements of the form
\begin{equation}\delta\frac{da^{e_1}}{a^{e_1}}\wedge \cdots \wedge\frac{da^{e_{n}}}{a^{e_{n}}}\label{e101}\end{equation}
with $\delta \in K$ and $e_i \in \Lambda_{s}$. Since logarithmic differential forms are linear, for each $e_i= (e_{i1},\ldots, e_{is})$,
\[\frac{da^{e_i}}{a^{e_i}} = \sum_{j=1}^{s}e_{ij}\frac{da_j}{a_j}\]
hence the $n$-forms in (\ref{e101}) are sums of $n$-forms of the form
\[\delta \frac{da_{j_1}}{a_{j_1}}\wedge\cdots\wedge\frac{da_{j_{n}}}{a_{j_{n}}}\]
for some $\delta \in K$. These forms are all zero, since $j_1,\ldots,j_{n}$ are chosen among $1,\ldots,s$ and $s < n$.
Therefore, $\Omega^{n}_L \to \Omega^{n}_K$ is the zero map.
\end{proof}

\section{Izhboldin's Filtration}
\label{s3}
Let $F$ be a field of characteristic $p$. Recall (\cite{I}) that the $p$-cohomology of $F$ is defined as 
\begin{equation}\H^{n+1}_p(F) = \textrm{coker}\left(\Omega^n_F \stackrel{\wp}{\longrightarrow} \Omega^n_F/d(\Omega^{n-1}_F)\right)
\label{eqn87}
\end{equation}
where for $a\in F$, $b_i \in F^*$, $\wp$ satisfies $\wp(a\frac{db_1}{b_1}\wedge\cdots\wedge\frac{db_n}{b_n}) = (a^p-a)\frac{db_1}{b_1}\wedge\cdots\wedge\frac{db_n}{b_n}$. We follow the convention of denoting an element of $\H^{n+1}_p(F)$ by an $n$-form to reduce notation.

In \cite{I} Izhboldin gives a filtration on the $p$-cohomology of $F$ where $F$ is a characteristic $p$ field complete with respect to a discrete valuation and residue field $\overline F$. We will heavily rely on this filtration and so we review it here. Given an integer $m$, $U_m = U_m\H^{n+1}_p(F)$ is defined to be the subgroup of $\H^{n+1}_p(F)$ generated by elements of the form
\[f\frac{dg_1}{g_1}\wedge\cdots\wedge\frac{dg_n}{g_n} \qquad \textrm{ with }\qquad f \in F, g_i \in F^*, v(f)\geq -m.\]
By \cite[3.3]{I} $U_{-1} = 0$ and by \cite[2.6]{I} if $F^{\textrm{ur}}$ is the maximal unramified extension of $F$ then $U_0 = \H^{n+1}_\ur(F)$ where $\H^{n+1}_\ur(F) = \ker (\H^{n+1}_p(F) \to \H^{n+1}_p(F^{\textrm{ur}}))$. Quotients of the filtration are understood by the following theorem. 
\begin{theorem}[\cite{I} Theorem 2.5] 
\[ U_i/U_{i-1} \cong \left\{\begin{array}{ll}
\H^{n+1}_p(\overline F) \oplus \H^n_p(\overline F) & \textrm{ if } i=0\\
\Omega^n_{\overline F} & \textrm{ if } i>0,\,p\nmid i\\
\Omega^n_{\overline F}/\Omega^n_{\overline F,d=0} \oplus \Omega^{n-1}_{\overline F}/\Omega^{n-1}_{\overline F,d=0} & \textrm{ if } i>0,\,p\mid i
\end{array}\right.\]
\label{t3}
\end{theorem}

\noindent
The isomorphisms in \ref{t3} will be denoted by $\rho_i^{-1}$ as is done in \cite[2.4]{I}. $\rho_0^{-1}:U_0 \to \H^{n+1}_p(\overline F) \oplus \H^n_p(\overline F)$ defines two maps, $\partial_1$ and $\partial_2$, the so-called {\it first and second residues}. For $a, b, b_i, c_i \in R^\times$ and $\pi \in R$ a fixed uniformizer for $F$, $\partial_1$ and $\partial_2$ are given by 

\begin{eqnarray}
\partial_1\left(a\frac{db_1}{b_1}\wedge \cdots \wedge \frac{db_n}{b_n}+b\frac{dc_1}{c_1}\wedge \cdots \wedge \frac{dc_{n-1}}{c_{n-1}}\wedge\frac{d\pi}{\pi}\right)&=&\left(\bar a\frac{d\bar b_1}{\bar b_1}\wedge \cdots \wedge \frac{d\bar b_n}{\bar b_n},0\right)\label{e62}
\\
\partial_2\left(a\frac{db_1}{b_1}\wedge \cdots \wedge \frac{db_n}{b_n}+b\frac{dc_1}{c_1}\wedge \cdots \wedge \frac{dc_{n-1}}{c_{n-1}}\wedge\frac{d\pi}{\pi}\right) &=&\left(0,\bar b\frac{d\bar c_1}{\bar c_1}\wedge \cdots \wedge \frac{d\bar c_{n-1}}{\bar c_{n-1}}\right)\label{e63}
\end{eqnarray}

\begin{remark} In \cite{I} the position of $d\pi/\pi$ in the definition of the first and second residues is in the first slot instead of last slot as above. This will possibly change the sign of the residues, but will not affect the isomorphisms.
\end{remark}

The following lemma describes how the isomorphisms in Theorem \ref{t3} behave with respect to scalar extensions. Let $e,\,m, \,n$ be positive integers. Let $F_1$ be a field of characteristic $p$ which is complete with respect to a discrete valuation $v$ and let $F_2$ be a complete subfield on which the valuation is non-trivial with ramification index $e$. Within the filtration on $\H^{n+1}_p(F_2) \to \H^{n+1}_p(F_1)$ there is a well defined \textit{extension of scalars map} $U_m/U_{m-1}(F_2) \to U_{em}/U_{em-1}(F_1)$ (since $e(m-1)\leq em-1$) which behaves as follows.

\begin{lemma}
Let $e,\,m,\,n,\,F_1$ and $F_2$ be as above. Let $\pi \in F_1$ and $\tau \in F_2$ be uniformizers with $\tau = u\pi^e$ and $u$ a unit. To reduce notation in the commutative diagrams below we use $\omega_n$ to indicate both a $n$-form in $\Omega^n$ and the class of that $n$-form in a quotient.
 \begin{enumerate}
\item If $p\nmid em$ then there is a commutative diagram
\[\xymatrix{
U_{em}/U_{em-1}(F_1) \ar[r]^(.7){\rho_{em}^{-1}} & \Omega^n_{\overline F_1}\\
U_m/U_{m-1}(F_2) \ar[r]^(.7){\rho_{em}^{-1}}\ar[u]^{\mathrm{res}} & \Omega^n_{\overline{F}_2}\ar[u]_{\psi_m}\\
}
\]
in which 
\[\psi_m:\omega_n \mapsto \bar u^{-m}\omega_n.\]
\label{e66}

\item If $m>0$ and $p|m$ then there is a commutative diagram

\[\xymatrix{
U_{em}/U_{em-1}(F_1) \ar[r]^(.35){\rho_{em}^{-1}} & \Omega^n_{\overline{F}_1}/\Omega^n_{\overline{F}_1,d=0}\oplus \Omega^{n-1}_{\overline{F}_1}/\Omega^{n-1}_{\overline{F}_1,d=0}\\
U_m/U_{m-1}(F_2) \ar[r]^(.35){\rho_{em}^{-1}}\ar[u]^{\mathrm{res}} & \Omega^n_{\overline{F}_2}/\Omega^n_{\overline{F}_2,d=0} \oplus \Omega^{n-1}_{\overline{F}_2}/\Omega^{n-1}_{\overline{F}_2,d=0}\ar[u]_{\psi_m}
}
\]
in which \[\psi_m: (\omega_n,\omega_{n-1}) \mapsto (\overline u^{-m}\omega_n+\omega_{n-1}\wedge\frac{d\bar u}{\bar u},e\,\bar u^{-m}\omega_{n-1}).\]
\label{ee66}

\item If $m=0$ then there is a commutative diagram
\[\xymatrix{
\H^{n+1}_\ur(F_1) \ar[r]^(.4){\rho_0^{-1}} & \H^{n+1}_p(\overline F_1)\oplus \H^{n}_p(\overline F_1)\\
\H^{n+1}_\ur(F_2) \ar[r]^(.4){\rho_0^{-1}}\ar[u]^{\mathrm{res}} & \H^{n+1}_p(\overline F_2) \oplus\H^{n}_p(\overline F_2)\ar[u]_{\psi_0} 
}
\]
in which \[\psi_0:(\omega_n,\omega_{n-1}) \mapsto (\omega_n + \omega_{n-1}\wedge \frac{d\bar u}{\bar u},e\,\omega_{n-1}).\]
\label{e45}
\end{enumerate}
\label{l4}
\end{lemma}
\begin{proof} Each of these is a diagram chase using the definitions for the maps $\rho_m$ from \cite[2.5]{I}. We illustrate the case $m=0$ here: let $\omega_i \in \Omega^i_{\overline F_2}$. Then $\rho_0$ of the class of $(\omega_n,\omega_{n-1})$ in $\H^{n+1}_p(\overline F_2)\oplus \H^{n}_p(\overline F_2)$ is the class of $\hat \omega_n+\hat \omega_{n-1}\wedge \frac{d\tau}{\tau}$ where $\hat \omega_i$ is any lift of $\omega_i$ to $F_2$. Extend scalars to $F_1$:
\begin{eqnarray*}\hat \omega_n+\hat \omega_{n-1}\wedge \frac{d\tau}{\tau} &=&\hat \omega_n+\hat \omega_{n-1}\wedge \frac{d(u\pi^e)}{u\pi^e}\\
&=&\hat \omega_n+\hat \omega_{n-1}\wedge \frac{du}{u}+e\omega_{n-1}\wedge \frac{d\pi}{\pi}
\end{eqnarray*}
Over $F_1$, $\rho_0^{-1}(\hat \omega_n+\hat \omega_{n-1}\wedge \frac{du}{u}+e\omega_{n-1}\wedge \frac{d\pi}{\pi})$ equals the class of $(\omega_n+\omega_{n-1}\wedge \frac{d\bar u}{\bar u}, e\omega_{n-1})$ in $\H^{n+1}_p(\overline F_1)\oplus \H^n_p(\overline F_1)$.
\end{proof}
\begin{remark} There is a similar commutative diagram for the case $p\nmid m$ and $p|e$, but we will not have the occasion to use it.
\end{remark}

Let $K/F$ be an extension of fields. In general the restriction map $\Omega^n_{F} \to \Omega^n_K$ is not an injection. For a simple example consider $\Omega^1_{k(x^p)} \to \Omega^1_{k(x)}$ which sends $0\ne d(x^p)$ to $d(x^p) = px^{p-1}dx = 0$. Sometimes $\Omega^n_{F} \to \Omega^n_K$ is an injection. For example purely transcendental extension fields $K/F$ give injections $\Omega^m_F \to \Omega^m_{K}$ (\cite[7.2]{MR3391933}) and  separable algebraic extensions $K/F$ give injections $\Omega^m_F \to \Omega^m_K$ (\cite[7.1]{MR3391933}). We will run into a case in the proof of Theorem \ref{t22} which also gives an injection, namely

\begin{lemma} Let $k$ be a perfect field of characteristic $p$ and let $E/k$ be a finitely generated extension with $p$-basis $\{a_1,\ldots, a_{r}\}$. Assume $K/E$ is a field extension with $p$-basis $\{a_1,\ldots, a_r,\ldots,a_s\}$ over $k$. Then for $n\geq 0$ the natural restriction maps 
\begin{eqnarray*}
\Omega^n_E &\to& \Omega^n_K\\
\Omega^n_E/\Omega^n_{E,d=0}
&\to& \Omega^n_K/\Omega^n_{K,d=0} 
\end{eqnarray*}
are injections.
\label{l78}
\end{lemma}

\begin{proof} Since $E$ has $p$-basis $\{a_1,\ldots,a_r\}$ it has differential basis $\{da_1,\ldots,da_r\}$ over $E$ and $\Omega^n_E$ has basis $\{da_{i_1}\wedge\cdots \wedge da_{i_n}\}_{i_1<\cdots<i_n}$ with $1\leq i_j\leq r$.  $K$ has differential basis $\{da_1,\ldots,da_s\}$ over $K$ and $\Omega^n_K$ has $K$-basis $\{da_{i_1}\wedge\cdots da_{i_n}\}_{i_1<\cdots<i_n}$ with $1\leq i_j\leq s$. The extension of these differential bases gives us the injections $\Omega^n_E \to \Omega^n_K$. For the second map, the injection $\Omega^{n}_E \to \Omega^{n}_K$ tells us that if $d(\omega)=0$ in $\Omega^{j}_K$, then $d(\omega)=0$ in $\Omega^{n}_E$.
\end{proof}

Consider now the complete case $F \cong K((\pi))$ for a field $K$ of characteristic $p$ which has finite $p$-rank. We want to write elements of $\H^{n+1}_p(K((\pi)))$ in a unique way using Izhboldin's $U_i$ filtration. $\Omega^{n}_K$ is a finite dimensional $K$-vector space, hence also a finite dimensional $K^p$-vector space. Fix $n\geq 0$ and $\{\nu_i\}_{i \in I_n}$ a $K^p$-basis for $\Omega^n_K$.  The cycle subset $\Omega^{n}_{K,d=0}$ is not a $K$-vector subspace of $\Omega^{n}_K$, but it is a $K^p$-vector subspace, i.e., if $d\omega = 0$ then for any $x \in K$, $d(x^p\omega) = x^pd\omega = 0$. Therefore there exists a subset $I'_n\subset I_n$ so that the image of  $\nu_i$ for $i \in I'_n$ is a $K^p$-basis for the quotient space $\Omega^{n}_K/\Omega^{n}_{K,d=0}$. Similarly fix $\{\omega_i\}_{i \in I_{n-1}}$, a $K^p$-basis for $\Omega^{n-1}_K$ and $I'_{n-1}\subset I_{n-1}$ a subset so that the images of the $\{\omega_i\}$ with ${i \in I'_{n-1}}$ form a $K^p$-basis of $\Omega^{n-1}_K/\Omega^{n-1}_{K,d=0}$.
\begin{lemma} Let $f \in \H^{n+1}_p(K((\pi)))$ and fix $K^p$-bases $\{\nu_i\}_{i \in I_n}$ and $\{\omega_j\}_{j \in I_{n-1}}$ of $\Omega^n_K$ and $\Omega^{n-1}_K$ as above. There exist unique $\alpha_{ki}, \,\beta_{ki},\,\gamma_{kj} \in K$ so that $f = \sum_{k=0}^m h_k$ with $h_0 \in U_0$ and for $k>0$
\begin{eqnarray*}
p\nmid k &:& h_k = \sum_{i \in I_n} \frac{\alpha_{ki}^p}{\pi^k}\nu_i\\
p|k &:& h_k = \sum_{i \in I'_n}\frac{\beta_{ki}^p}{\pi^k}\nu_i + \sum_{j \in I'_{n-1}}\frac{\gamma_{kj}^p}{\pi^k}\omega_j\wedge\frac{d\pi}{\pi}  
\end{eqnarray*}
Moreover each  $h_k \in U_k(K((\pi)))$.
\label{l1}
\end{lemma}

\begin{proof} If $f \in U_0$ then $\alpha_{ki}=\beta_{ki}=\gamma_{kj}=0$ gives  a solution. Let $\alpha_{ki}'$, $\beta_{ki}'$, $\gamma_{ki}'$ be another choice of coefficients and let $m$ be the maximum integer with one of $\alpha_{mi}'$, $\beta_{mi}'$ or $\gamma_{mi}'$ nonzero. If $m>0$ then by our choice of bases, $\rho_m^{-1}(f) \ne 0$ (Theorem \ref{t3}). This contradicts that $f \in U_0\subset U_{m-1}$.  

Assume $f \notin U_0$ and let $m$ be the minimum integer with $f \in U_m$. Consider the image of $f$ in $U_m/U_{m-1}$. Use the isomorphisms in Theorem \ref{t3} together with $\overline {K((\pi))} \cong K$ to find the unique coefficients $\alpha_{ki}, \,\beta_{ki},\,\gamma_{kj} \in K$ which satisfy $f-\sum_{i \in I_n} \frac{\alpha_{ki}^p}{\pi^k}\nu_i \in U_{m-1}$ if $p\nmid k$ and $f-\sum_{i \in I'_n}\frac{\beta_{ki}^p}{\pi^k}\nu_i + \sum_{j \in I'_{n-1}}\frac{\gamma_{kj}^p}{\pi^k}\omega_j\wedge \frac{d\pi}{\pi} \in U_{m-1}$ if $p|k$. Apply induction to the new element.
\end{proof}

In Theorem \ref{t22} we will be given classes in $\H^{n+1}(K((\pi)))$ which are not quite in the canonical form of Lemma \ref{l1}. We will need to put them in canonical form and   determine what happens to the $U_0$ term in the process. The answer is the $U_0$ terms stays the same and the proof will use the following equality in $\H^{n+1}_p(K((\pi)))$: for $N \in \Z$ with $p\nmid N$ and any $\omega \in \Omega^{n}_{K((\pi))}$ we have
\begin{eqnarray}
\frac{\omega}{\pi^N}\wedge\frac{d\pi}{\pi} &=& \frac{\omega}{\pi^N}\wedge\frac{d\pi}{\pi} + d\left(\frac{N^{-1}\omega}{\pi^N}\right)\nonumber \\
&=& \frac{N^{-1}d\omega}{\pi^N}.
\label{eq67}
\end{eqnarray}

\begin{proposition} Let $f \in \H^{n+1}_p(K((\pi)))$ be an element of the form $f = \sum_{r=0}^N f_r$ where $f_0 \in U_0$ and for $r>0$
\[f_r = \frac{g_{r}}{\pi^r} + \frac{g'_{r}}{\pi^r}\wedge\frac{d\pi}{\pi}\]
with $g_{r} \in \Omega^{n}_K$ and $g'_{r} \in \Omega^{n-1}_K$. Then, when we write $f$ in its canonical form $f = \sum_{k=0}^m h_k$ as in Lemma \ref{l1}, $h_0 = f_0$. In particular, if $f \in U_0(K((\pi)))$ then $f = f_0$.
\label{p2}
\end{proposition}

\begin{proof} We proceed by induction on $N$. If $N=0$ then $f$ is already in canonical form and there is nothing to prove. Fix $N>0$ and assume the proposition is true for all $N_0<N$. Let $f=f_N+\cdots+f_0$ with the $f_r$'s as in the statement of the proposition. If $p\nmid N$, then by (\ref{eq67}):
\begin{eqnarray*}
f_N &=& \frac{g_N}{\pi^N} + \frac{g'_N}{\pi^N}\wedge\frac{d\pi}{\pi}\\
&=&\frac{g_N}{\pi^N} + \frac{N^{-1}dg'_N}{\pi^N}
\end{eqnarray*}
Write $g_N + N^{-1}dg'_N \in \Omega^n_{K}$ as $\sum_{i \in I_n} \alpha_{Ni}^p\nu_i$ with $\alpha_{Ni} \in K$. Then
\[f_N = \sum_{i \in I_n}\frac{\alpha_{Ni}^p}{\pi^N}\nu_i\]
is in canonical form. The result holds by induction on $f-f_N$.  If $p|N$ write
\begin{eqnarray*}
g_N &=& \sum_{i \in I'_n}\beta_{Ni}^p\nu_i + \mu_n\\
g'_N &=& \sum_{j \in I'_{n-1}}\gamma_{Nj}^p\omega_j + \mu_{n-1}
\end{eqnarray*}
with $\beta_{Ni},\,\gamma_{Nj} \in K$, $\mu_n \in \Omega^{n}_{K,d=0}$ and $\mu_{n-1} \in \Omega^{n-1}_{K,d=0}$.
By Cartier's isomorphism \cite[1.5.3]{I} $d(\mu_i) = 0$ implies  $\mu_i = \Phi(\epsilon_i)+d(\xi_i)$ for some $\epsilon_i \in \Omega^{i}_K$, $\xi_i \in \Omega^{i-1}_K$. Here $\Phi:\Omega^i_K \to \Omega^i_K$ is the Frobenius homomorphism 
\[\Phi\,\,:\,\,a\frac{db_1}{b_1}\wedge \cdots \wedge \frac{db_i}{b_i} \,\,\longrightarrow\,\, a^p\frac{db_1}{b_1}\wedge \cdots \wedge \frac{db_i}{b_i}.\]
All together we have 
\[f_N = \sum_{i \in I'_n}\frac{\beta_{Ni}^{p}}{\pi^N}\nu_i  +\frac{d(\xi_n)}{\pi^N} + \frac{\Phi(\epsilon_n)}{\pi^N}  + \sum_{j \in I'_{n-1}} \frac{\gamma_{Nj}^{p}\omega_j}{\pi^N}\wedge\frac{d\pi}{\pi} + \frac{d(\xi_{n-1})}{\pi^N}\wedge\frac{d\pi}{\pi}+ \frac{\Phi(\epsilon_{n-1})}{\pi^N}\wedge\frac{d\pi}{\pi} \]
We need to re-write the 2nd, 3rd, 5th and 6th terms in this sum while the 1st and 4th terms are already in canonical form. We deal with the 2nd and 5th terms similarly; since $p|N$, $d(\xi_n)/\pi^N = d(\xi_n/\pi^N)$ and $\frac{d(\xi_{n-1})}{\pi^N}\wedge\frac{d\pi}{\pi} = d(\xi_{n-1}/\pi^N)\wedge\frac{d\pi}{\pi} = d\left(\frac{\xi_{n-1}}{\pi^N}\wedge\frac{d\pi}{\pi}\right)$. Since $d(-)=0$ in $\H^2_p(K((\pi)))$ we can replace both of these terms by 0. The 3rd and 6th terms are also similar; by (\ref{eqn87})  we have $\Phi(\epsilon_n)/\pi^N = \epsilon_n/\pi^{N/p}$ and 
\[\frac{\Phi(\epsilon_{n-1})}{\pi^N}\wedge\frac{d\pi}{\pi} = \frac{\epsilon_{n-1}}{\pi^{N/p}}\wedge\frac{d\pi}{\pi}\]
in $\H^2_p(K((\pi)))$. The 6 terms in $f_N$ have turned into:
\begin{eqnarray*}
f_N &=& \sum_{i \in I'_n}\frac{\beta_{Ni}^{p}}{\pi^N}\nu_i  +0+ \frac{\epsilon_n}{\pi^{N/p}}  +\sum_{j \in I'_{n-1}} \frac{\gamma_{Nj}^{p}\omega_j}{\pi^N}\wedge\frac{d\pi}{\pi} + 0+\frac{\epsilon_{n-1}}{\pi^{N/p}}\wedge\frac{d\pi}{\pi}\\
&=& h_N + \frac{\epsilon_n}{\pi^{N/p}} + \frac{\epsilon_{n-1}}{\pi^{N/p}}\wedge\frac{d\pi}{\pi}
\end{eqnarray*}
with $h_N$ in canonical form. Moreover,  
\[f = h_N +f_{N-1} +\ldots+f_{N/p}+\frac{\epsilon_n}{\pi^{N/p}} + \frac{\epsilon_{n-1}}{\pi^{N/p}}\wedge\frac{d\pi}{\pi}+\ldots + f_0\]
where each of $f_i$, $i\ne N/p$ and $f_{N/p} + \frac{\beta_1}{\pi^{N/p}} + \frac{\beta_0}{\pi^{N/p}}\wedge\frac{d\pi}{\pi}$ are as in the statement of the proposition. Note in particular, $N/p \ne 0$, so that we did not alter $f_0$. Apply the induction hypothesis to $f-h_N$ to finish the proof. The last sentence follows because $f\in U_0(K((\pi)))$ is already in canonical form.
\end{proof}

Let $K$ be a discrete valued field of characteristic $p$ with uniformizer $\pi$ and residue field $\overline K$. In Theorem \ref{t22} we will be looking at $n$-forms coming from subfields of the form $L=K^p(a_1,\ldots, a_s)\subset K = K^p(a_1,\ldots, a_r)$ where $s\leq r$ and $\{a_i\}_{i=1}^r$ is a differential basis for $K/k$ coming from $\overline K$. Let $a_{i_0} = \pi$ be the uniformizer for $K$ in this differential basis so that the completion $\widehat K \cong K_1((\pi))$ and the coefficient field $K_1$ contains all $a_i$ with $i \ne i_0$ i.e., all those with $v(a_i)=0$ (\cite[7.8]{E}).

\begin{lemma} Let $K$, $\{a_i\}_{i=1}^r$, $L$, $K_1$ and $\pi$ be as above. Let $g \in \H^{n+1}_p(L)$. Then upon extension of scalars to $\widehat K$
\[g_{\widehat K}  = \res{\widehat K}(g)= g_m+\cdots+g_0\]
where each $g_i\in \H^{n+1}_p(\widehat K)$ is a sum of elements of the form $\frac{f}{\pi^i}\frac{da^{e_1}}{a^{e_1}}\wedge \cdots \wedge\frac{da^{e_{n}}}{a^{e_{n}}}$ with $f \in K_1$, $e_i \in \Lambda_{s}$. Moreover, 
\begin{enumerate}
\item if $i_0>s$ then $\partial_2(g_0) = 0$
\item for any discrete valuation $w$ on $K_1$ with uniformizer $\tau$ and residue field $\overline{K}_1$, there exists a differential basis $B' = \{a_1',\ldots,a_{r-1}'\}$ for $K_1/k$  coming from $\overline{K}_1$ so that  
\begin{enumerate}
\item if $i_0>s$ then $K_1^p(a_1,\ldots,a_{s}) = K_1^p(a_1',\ldots,a_{s}')$ and $\partial_1(g_0)$ descends to\\ $K_1^p(a_1',\ldots, a_{s}')$,
\item if $i_0=s$ then $K_1^p(a_1,\ldots,a_{s-1}) = K_1^p(a_1',\ldots,a_{s-1}')$ and $\partial_1(g_0)$ descends to $K_1^p(a_1',\ldots, a_{s-1}')$.\end{enumerate}
\end{enumerate}
\label{lh2}
\end{lemma}

\begin{remark} In the statement of Lemma \ref{lh2} we are identifying the coefficient field $K_1$ (and hence also $K_1^p(a_1,\ldots, a_s)$) with the residue field of $\widehat K$. In this way the statements ``$\partial_1(g_0)$ descends to $K_1^p(a_1',\ldots)$'' in Lemma \ref{lh2} make sense.
\end{remark}

\begin{proof}As in Lemma \ref{l999} we consider the extension of scalars map: $\Omega^{n}_L \to \Omega^{n}_K$. Since $s$ is not necessarily less than $n$, the map may be nonzero, but we can still express the $n$-forms using the $K^p$-basis of $L$. In particular, given $bdc_{1}\wedge\cdots\wedge dc_{n}\in \Omega^{n}_L$ write 
\[b= \sum_{  e \in \Lambda_{s}} \beta_{e}^pa^{e} \hspace{.25in} \textrm{and} \hspace{.25in}
c_{i} = \sum_{  e' \in \Lambda_{s}} \gamma_{ie'}^pa^{ e'}\]
with $\beta_{e}, \gamma_{ie'} \in K$. As in Lemma \ref{l999} we use these expressions for $b$ and $c_{i}$ and extend scalars from $L$ to $K$ to expand  $bdc_{1}\wedge\cdots\wedge dc_{d-1}$, but this time we are a bit more careful and expand it into a sum of elements of the form
\begin{equation}\frac{\delta^pa^e}{\pi^{pk}}\frac{da^{e_1}}{a^{e_1}}\wedge \cdots \wedge\frac{da^{e_{n}}}{a^{e_{n}}}\label{e99}\end{equation}
with $\delta \in R^\times$, $k \in \Z$ and $e,e_i \in \Lambda_{s}$. Now extend scalars further to $g_{\widehat K}$ and note that if $k<0$, then $\frac{\delta^pa^e}{\pi^{pk}}\frac{da^{e_1}}{a^{e_1}}\wedge \cdots \wedge\frac{da^{e_{d-1}}}{a^{e_{d-1}}}\in U_{-1}(\widehat K)\subset \H^2_p(\widehat K)$ which is zero by (\cite[3.3]{I}).

We can thus simplify elements of the form (\ref{e99}) over $\widehat K$ by expressing $\delta = f_0+f_1\pi+\cdots + f_{k}\pi^k+f'\pi^{k+1}$ with $f_i \in K_1$ and $f' \in R^\times$, so that  
\begin{eqnarray}
\frac{\delta^pa^e}{\pi^{pk}}\frac{da^{e_1}}{a^{e_1}}\wedge \cdots \wedge\frac{da^{e_{n}}}{a^{e_{n}}}&=& \sum_{i=0}^k\frac{f_i^pa^e}{\pi^{p(k-i)}}\frac{da^{e_1}}{a^{e_1}}\wedge \cdots \wedge\frac{da^{e_{n}}}{a^{e_{n}}}+\,\,(f')^p\pi^p a_e\frac{da^{e_1}}{a^{e_1}}\wedge \cdots \wedge\frac{da^{e_{n}}}{a^{e_{n}}}\nonumber\\
&=&\sum_{i=0}^k\frac{f_i^pa^e}{\pi^{p(k-i)}}\frac{da^{e_1}}{a^{e_1}}\wedge \cdots \wedge\frac{da^{e_{n}}}{a^{e_{n}}}
\label{eq86}
\end{eqnarray}

If $i_0>s$ then $a^e \in K_1$ for all $e \in \Lambda_s$. If $i_0\leq s$ then we will assume $i_0=s$ (after reordering if necessary) and if $e=(\epsilon_1,\ldots,\epsilon_s)$ then $a^e\pi^{-\epsilon_s} \in K_1$. In both cases, since $f_i\in K_1$, we've shown the class of $g_{\widehat K}$ can be written as
\begin{equation}g_{\widehat K} = g_m+\cdots +g_{0}\label{e100}\end{equation}
where each $g_i$ is a sum of elements of the form $\frac{f}{\pi^i}\frac{da^{e_1}}{a^{e_1}}\wedge \cdots \wedge\frac{da^{e_{n}}}{a^{e_{n}}}$ with $f \in K_1$ and $e_i \in \Lambda_s$. 

To show the final part of the lemma, return to (\ref{eq86}) and consider those terms contributing to the $g_0$ component of $g_{\widehat K}$. If $i_0> s$ then a term of the form (\ref{eq86}) contributes to $g_0$ only if $p(k-i) = 0$ and it is then immediate that both $\partial_2(g_0)=0$ (there are no uniformizers in the wedge product) and $g_0$ descends to $K_1^p(a_1,\ldots, a_s)$.  

If $i_0= s$ and the term
\[\frac{f_i^p}{\pi^{p(k-i)}}a^e\frac{da^{e_1}}{a^{e_1}}\wedge \cdots \wedge\frac{da^{e_{n}}}{a^{e_{n}}}\]
with $e=(\epsilon_1,\ldots,\epsilon_s)$ contributes to the $g_0$ piece then $\epsilon_{s}-p(k-i) = 0$. In particular, $p|\epsilon_s$ and thus $\epsilon_s=0$. The contributing element must therefore look like
\[f_i^pa^e\frac{da^{e_1}}{a^{e_1}}\wedge \cdots \wedge\frac{da^{e_{n}}}{a^{e_{n}}} \]
with $\epsilon_s=0$. Set $e_i = (\epsilon_{i1},\ldots,\epsilon_{is})$ and separate out the uniformizers in the wedge product:
\begin{eqnarray*}
f_i^pa^e\frac{da^{e_1}}{a^{e_1}}\wedge \cdots \wedge\frac{da^{e_{n}}}{a^{e_{n}}}&=& f_i^pa^e\left(\frac{d(a^{e_1}\pi^{-\epsilon_{1s}})}{a^{e_1}\pi^{-\epsilon_{1s}}} + \epsilon_{1s}\frac{d\pi}{\pi}\right)\wedge \cdots \wedge \left(\frac{d(a^{e_{n}}\pi^{-\epsilon_{ns}})}{a^{e_{n}}\pi^{-\epsilon_{ns}}} + \epsilon_{ns}\frac{d\pi}{\pi}\right)\\
&=&{\omega_i} + {\nu_i}\wedge\frac{d\pi}{\pi}.
\end{eqnarray*}
with $\omega_i \in \Omega^{n}$ and $\nu_i \in \Omega^{n-1}$  forms over $K_1^p(a_1,\ldots,a_{s-1})$. In particular, $g_0 = \omega+\nu\wedge\frac{d\pi}{\pi}$ with $\omega$ an $n$-form defined over $K_1^p(a_1,\ldots,a_{s-1})\subset K_1$. By construction we can identify $K_1$ with the residue field of $\widehat K$ and thus also $K_1^p(a_1,\ldots, a_{s-1})$ as a subfield of $\widehat K$. In particular, since $\partial_1(g_0)=\omega$, $\partial_1(g_0)$ descends to $K_1^p(a_1,\ldots,a_{s-1})$. 

Finally, let $w$ be a discrete valuation on $K_1$ with uniformizer $\tau$ and residue field $\overline K_1$. Arguing as in the proof of Proposition \ref{p1}, if $i_0>s$ (resp. $i_0 = s$) then there exists a differential basis $B' = \{a_1',\ldots,a_{r-1}'\}$ for $K_1/k$ coming from $\overline K_1$ so that $K_1^p(a_1,\ldots,a_{s}) = K_1^p(a_1',\ldots, a_{s}')$ (resp. $K_1^p(a_1,\ldots,a_{s-1}) = K_1^p(a_1',\ldots,a_{s-1}')$). Since we've already shown $\partial_1(g_0)$ to descend appropriately, this finishes the proof.
\end{proof}

\section{Essential dimension of the generic symbol}
\label{s4}
We now look at generic symbols in characteristic $p$. Fix integers $\ell, n\geq 1$ and $k$  an algebraically closed field of characteristic $p$. Set 
\[k_{\ell,n} = k(x_i,y_{i,j})_{1\leq i \leq \ell,1\leq j \leq n}\]
the rational function field defined by $\ell(n+1)$ independent variables over $k$.  Denote by $\gen_k(n+1,\ell,p)$ the $\H^{n+1}_p(k_{\ell,n})$ class of the length $\ell$ generic $p$-symbol of degree $n+1$ over $k$, i.e., 
\[\gen_k(n+1,\ell,p) = \sum_{i=1}^\ell x_i\frac{dy_{i,1}}{y_{i,1}}\wedge\cdots\wedge \frac{dy_{i,n}}{y_{i,n}} \in \H^{n+1}_p(k_{\ell,n}).\]
Throughout this section let $v_{{\ell,n}}$ denote the $y_{\ell,n}$-adic valuation on $k_{\ell,n}$, $\widehat k_{\ell,n}$ the completion and $\overline k_{\ell,n}$ the corresponding residue field. Note that $\gen_k(n+1,\ell,p) \in U_0(\widehat k_{\ell,n})$ and therefore we can look at its first and second residues. 
\begin{lemma}
Fix an isomorphism $\overline k_{\ell,n} \cong k_{\ell-1,n}(x_\ell,y_{\ell,1},\ldots,y_{\ell,n-1})$ and inclusions $k_{i,j}\subset \overline k_{\ell,n}$ for all $i\leq \ell$ and $j < n$ or $i<\ell$ and $j\leq n$.  Over $\widehat k_{\ell,n}$ and with respect to the uniformizer $y_{\ell,n}$  
\begin{eqnarray*}
\partial_1(\gen_k(n+1,\ell,p)) &=& \res{\overline k_{\ell,n}}(\gen_k(n+1,\ell-1,p))\\
\partial_2(\gen_k(n+1,\ell,p)) &=& \res{\overline k_{\ell,n}}\left(x_\ell\frac{dy_{\ell,1}}{y_{\ell,1}}\wedge \cdots \wedge \frac{dy_{\ell,n-1}}{y_{\ell,n-1}}\right).\\
\end{eqnarray*}
\label{l2}
\end{lemma}
\begin{proof}This follows from the description of the residues in (\ref{e62}) and (\ref{e63}).
\end{proof}

\begin{lemma}
Let $n\geq 1$, $\ell\geq 1$, and let $\kln\subset \kln(z_1,\ldots,z_r)\subset K$ be fields with $K/\kln(z_1,\ldots,z_r)$ a prime to $p$ extension and the $z_i$'s algebraically independent indeterminates over $\kln$. Then for any integer $e$ which is prime to $p$, 
\[e\cdot\res{K/k_{\ell,n}}(\gen_k(n+1,\ell,p)) \ne 0.\]
\label{c11}
\end{lemma}

\begin{proof} We proceed by induction on $n$. When $n=1$, we can reduce notation a bit by setting $y_{i1} = y_i$, so that the field $k_{\ell,1} = k(x_1,y_1,\ldots,x_\ell,y_\ell)$ and the element $\gen(2,\ell,p)$ corresponds to the 1-form 
\[\gen(2,\ell,p) = \sum_{i=1}^\ell x_i\frac{dy_{i}}{y_{i}}\]
Under the isomorphism $\H^2_p(k_{\ell,1}) \cong {}_p\br(k_{\ell,1})$ (\cite[9.2.5]{GS}) $\gen(2,\ell,p)$ maps to the class of the length $\ell$ generic $p$-symbol algebra $\otimes_{i=1}^\ell[x_i,y_{i})$. The index of $\otimes_{i=1}^\ell[x_i,y_{i})$ is $p^\ell$ and the exponent is $p$ which can be seen via generic abelian crossed product $p$-algebras (\cite[2.7]{MR857567} or \cite[p.4]{MR1923420}).  A purely transcendental extension $k_{\ell,1}\subset k_{\ell,1}(z_1,\ldots,z_r)$, a prime to $p$ extension $K$ and multiplication by $e$ all give injections of the $p$-torsion part of the Brauer group (\cite{MR1692654}), proving the result in this case.

Fix $n>1$ and assume the theorem holds for $\gen_k(n_0+1,\ell,p)$ for all $1\leq n_0<n$, $\ell\geq 1$ and  fields $k$ of characteristic $p$. Let $K/k_{\ell,n}$  and $e$ be as in the statement of the theorem. Choose a valuation $v$ on $K$ which extends $v_{{\ell,n}}$, the  $y_{\ell,n}$-adic valuation on $k_{\ell,n}(z_1,\ldots,z_r)$ so that both the residue degree $f(v/v_{{\ell,n}})$ and ramification degree $e(v/v_{{\ell,n}})$ are prime to $p$ (see Example \ref{ex1}). Let $\widehat K$ and $\widehat k_{\ell,n}(z_1,\ldots,z_r)$ be the respective completions and consider the second residue maps, $\partial_2$, on these fields.  By Lemma \ref{l4}(\ref{e45}) we have
\[\xymatrix{
\H^{n+1}_p(K)\ar[r]^{\textrm{res}}&\H^{n+1}_\ur(\widehat K)\ar[r]^{\partial_2}& \H^{n}_p(\overline K)\\
&\H^{n+1}_\ur(\widehat k_{\ell,n}(z_1,\ldots,z_r))\ar[r]^{\partial_2}\ar[u]^{e\cdot \textrm{res}}&\H^{n}_p(\overline{k_{\ell,n}(z_1,\ldots,z_r)})\ar[u]_{e\cdot e(v/v_{{\ell,n}})\cdot\textrm{res}}\\
\H^{n+1}_p(k_{\ell,n})\ar[r]^{\textrm{res}}\ar[uu]^{e\cdot \textrm{res}}&\H^{n+1}_\ur(\widehat k_{\ell,n})\ar[r]^{\partial_2}\ar[u]^{\textrm{res}}&\H^{n}_p(\overline k_{\ell,n})\ar[u]_{\textrm{res}}
}\]
Using Lemma \ref{l2} and tracing the diagram in both directions shows that if $e\cdot\res{K/k_{\ell,n}}(\gen_k(n+1,\ell,p)) =0$ then 
\begin{equation}e\cdot e(v/v_{{\ell,n}})\cdot\res{\overline K/\overline k_{\ell,n}}\left(x_\ell\frac{dy_{\ell,1}}{y_{\ell,1}}\wedge \cdots \wedge \frac{dy_{\ell,n-1}}{y_{\ell,n-1}}\right)=0.
\label{e77}
\end{equation}

The field extension $\overline K/\overline k_{\ell,n}$ can be decomposed into extensions
\[\xymatrix{
\overline K\\
\overline{k_{\ell,n}(z_1,\ldots,z_r)}\cong k_{\ell-1,n}(x_\ell,y_{\ell,1}\ldots,y_{\ell,n-1},z_1,\ldots,z_r)\ar@{-}[u]^{f(v/v_{{\ell,n}})}_{\textrm{  prime to }p}\\
\overline{k_{\ell,n}} \cong  k_{\ell-1,n}(x_\ell,y_{\ell,1}\ldots,y_{\ell,n-1})\ar@{-}[u]_{\textrm{purely transcendental}}\\
k(x_\ell,y_{\ell,1},\ldots,y_{\ell,n-1})\ar@{-}[u]_{\textrm{purely transcendental}}\\
}\]
In other words, the extension is a composition of purely transcendental extensions followed by a prime to $p$ extension. Since, up to numbering, $x_\ell\frac{dy_{\ell,1}}{y_{\ell,1}}\wedge\cdots\wedge\frac{dy_{\ell,n-1}}{y_{\ell,n-1}} = \gen_k(n,1,p)$, (\ref{e77}) violates the induction hypothesis.
\end{proof}

Let $k$ be an algebraically closed field of characteristic $p$. Let $K/ k_{\ell,n}(z_1,\ldots,z_r)$ be a finite prime to $p$ extension as above and fix the uniformizer $y_{\ell,n}$ for the valuation $v_{{\ell,n}}$ on $k_{\ell,n}(z_1,\ldots,z_r)$. As in the proof of Corollary \ref{c11} choose an extension $v$ of $v_{{\ell,n}}$ to $K$ with both $e(v/\vyln)$ and $f(v/\vyln)$ prime to $p$. Let $\widehat K$ and $\widehat k_{\ell,n}$ be the corresponding completions and $\overline K$ and $\overline k_{\ell,n}$ the residue fields with respect to these valuations.

\begin{corollary} Let $K/k_{\ell,n}$ be as above with valuations $v$ and $\vyln$. For any $n, \ell\geq 1$, $\partial_2(\gen_k(n+1,\ell,p)_{\widehat K})\ne 0$.
\label{cor56}
\end{corollary}
\begin{proof} By Lemma \ref{l2} the first and second residues of $\gen_k(n+1,\ell,p)_{\widehat k_{\ell,n}}$ are sums of generic $p$-symbols (at least after re-numbering the variables) with scalars extended to $\overline k_{\ell,n}$. Recall $\overline k_{\ell,n}$ is isomorphic to $k_{\ell-1,n}(x_\ell,y_{\ell,1},\ldots,y_{\ell,n-1})$ and therefore $\overline K$ is a prime to $p$ extension of a purely transcendental extension of $k(x_\ell,y_{\ell,1},\ldots,y_{\ell,n-1})$. By Lemma \ref{l4}(\ref{e45}) 
\[\partial_2(\gen_k(n+1,\ell,p)_{\widehat K}) = e(v/\vyln)\cdot\res{\overline K/k(x_\ell,y_{\ell,1},\ldots,y_{\ell,n-1})}\left(x_\ell\frac{dy_{\ell,1}}{y_{\ell,1}}\wedge\cdots\wedge\frac{dy_{\ell,n-1}}{y_{\ell,n-1}}\right)\]
The latter is nonzero by Lemma \ref{c11}.
\end{proof}

Following the notation in \cite{Reichstein-ICM}, for $D \in \H^{n+1}_p(K)$ we denote the {\it essential dimension} of $D$ as an element of $\H^{n+1}_p(K)$ over $k$ by $\ed_k(D)$ and the $p$-essential dimension by $\ed_k(D;p)$.  

\begin{lemma}
$\ed_k(\gen_k(n+1,1,p)) = \ed_k(\gen_k(n+1,1,p); p) = n+1$.\label{l3}
\end{lemma}

\begin{proof} The essential dimension is bounded above by $n+1$  since $\gen_k(n+1,1,p)$ is defined over $k_{1,n}=k(x_1,y_{11},\ldots, y_{1,n})$. For the lower bound, suppose there is a  prime-to-$p$-extension $K/k_{1,n}$, a field $k\subset E\subset K$ and a $g \in \H^{n+1}_p(E)$ so that $\res{K/E}(g)=\res{K/k_{1,n}}(\gen_k(n+1,1,p))$. If $\trdeg_k(E)<n+1$, then $E$ is $C_r$ for some $r<n+1$, hence $\H^{n+1}_p(E) = 0$ as in \cite{MR2600061}. This contradicts Lemma \ref{c11}.
\end{proof}

\begin{remark} The proof of Lemma \ref{l3} also shows that $\ed_k(\gen_k(n+1,\ell,p))\geq \ed_k(\gen_k(n+1,\ell,p);p)\geq n+1$. But this result will be subsumed by Theorem \ref{t1}.
\end{remark}

\begin{theorem}[Generalization of Babic \& Chernousov's 11.3] Let $\ell,n\geq 1$ and $(K,v)$ a valued prime to $p$ extension of $(k_{\ell,n}(z_1,\ldots,z_r),v_{{\ell,n}})$ with residue degree and ramification index prime to $p$. There does not exists a differential basis $B = \{a_1,\ldots,a_{\ell(n+1)+r}\}$ for $K/k$ coming from $\overline K$ such that $\res{K}(\gen_k(n+1,\ell,p))$ descends to $K^p(a_1,\ldots, a_{n+\ell-2})$. 
\label{t22}
\end{theorem}

\begin{proof}This proof follows the general outline of the proof of \cite[11.3]{bc} and proceeds by induction on the length of the generic symbol, $\ell$. {\it Case $\ell=1$}. 
Let $B = \{a_1,\ldots, a_{n+1+r}\}$ be a differential basis for $K/k$ coming from $\overline K$ and $g \in \H^{n+1}_p(L)$ with $L = K^p(a_1,\ldots, a_{n-1})$ so that $\res{K/L}(g) = \res{K}(\gen_k(n+1,1,p))$. By Lemma \ref{l999} the extension of scalars map $\Omega^{n}_L \to \Omega^{n}_K$ is the zero map, hence the extension of scalars $\H^{n+1}_p(L) \to \H^{n+1}_p(K)$ is also the zero map. This contradicts $\res{K}(\gen_k(n+1,1,p)) \ne 0$ (Lemma \ref{c11}).

Fix $\ell>1$ and assume the theorem holds for $\gen_k(n+1,\ell_0,p)$ with $1\leq\ell_0 <\ell$ for all algebraically closed fields $k$ of characteristic $p$, all $n\geq 1$ and all $r\geq 0$. Assume there exists a differential basis $B = \{a_1,\ldots, a_{\ell(n+1)+r}\}$ for $K/k$ coming from $\overline K$ and $g \in \H^{n+1}_p(L)$ with $L = K^p(a_1,\ldots, a_{\ell+n-2})$ such that $\res{K/L}(g) = \res{K}(\gen_k(n+1,\ell,p))$. Let $L$ and $K$ have completions $\widehat K$ and $\widehat L$ and residue fields $\overline K$ and $\overline L$ with respect to $v$. By \cite[7.8]{E} the differential basis $B$ for $K/k$ corresponds to a coefficient field $K_1\subset \widehat K$ containing $\{a_i \,|\,v(a_i) = 0\}$. Let $a_{i_0} = \pi$ be the uniformizer of $K$ in the differential basis so that $\widehat K \cong K_1((\pi))$ and $a_i \in K_1$ for all $i \ne i_0$.

Using Lemma \ref{lh2} we can write
\begin{equation}
g_{\widehat K}  = \res{\widehat K/L}g= g_m+\cdots+g_0
\label{eq66}
\end{equation}
where each $g_i$ is a sum of elements of the form $\frac{f}{\pi^i}\frac{da^{e_1}}{a^{e_1}}\wedge \cdots \wedge\frac{da^{e_{n}}}{a^{e_{n}}}$ with $f \in K_1$ and $e_i \in \Lambda_{\ell+n-2}$. We now consider the two cases $i_0>\ell+n-2$ and $i_0\leq \ell+n-2$ separately. 

If $i_0>\ell+n-2$ then $a^e \in K_1$ for all $e \in \Lambda_{\ell+n-2}$ and each $g_i$ in (\ref{eq66}) can be written as $g_i = \omega_i/\pi^i$ with $\omega_i \in \Omega^{n}_{K_1}$. In other words, $g_{\widehat K} = \sum_i \omega_i/\pi^i$ with $\omega_i \in \Omega^{n}_{K_1}$, so we can apply Proposition \ref{p2}. Since $g_{\widehat K} = \res{\widehat K}(\gen_k(n+1,\ell,p)) \in U_0(\widehat K)$, Proposition \ref{p2} says  $g_{\widehat K} = g_0$. Since $g_0 \in \Omega_{K_1}^{n}$, $0=\partial_2(g_{\widehat K}) = \partial_2(\gen(n+1,\ell,p)_{\widehat K})$.  This contradicts Corollary \ref{cor56}.

When $i_0 \leq \ell+n-2$ we derive a contradiction using the induction hypothesis. Assume $i_0 = \ell+n-2$ (after re-ordering if necessary). As in (\ref{eq66}) we can write $g_{\widehat K} = \sum g_i$ where $g_i$ are homogeneous with terms of the form 
\begin{equation}\frac{f}{\pi^i}\frac{da^{e_1}}{a^{e_1}}\wedge \cdots \wedge\frac{da^{e_{n}}}{a^{e_{n}}}.
\label{eq888}
\end{equation}
with $f \in K_1$ and $e_i \in \Lambda_{\ell+n-2}$. Using $a_{\ell+n-2} = a_{i_0}= \pi$, we want to separate out the uniformizers in the logarithmic differentials in (\ref{eq888}) as in the proof of Lemma \ref{lh2}. Set $e_i = (\epsilon_{i,1},\ldots,\epsilon_{i,i_0}) \in \Lambda_{\ell+n-2}$. Then
\[\frac{da^{e_i}}{a^{e_i}} = \frac{da^{e_i}\pi^{-\epsilon_{i,i_0}}}{a^{e_i}\pi^{-\epsilon_{i,i_0}}}+\epsilon_{i,i_0}\frac{d\pi}{\pi}\]
and $a^{e_i}\pi^{-\epsilon_{i,i_0}} \in K_1$ for all $i$. The terms in (\ref{eq888}) become 
\begin{eqnarray*}
\frac{f}{\pi^i}\frac{da^{e_1}}{a^{e_1}}\wedge \cdots \wedge\frac{da^{e_{n}}}{a^{e_{n}}}&=& \frac{f}{\pi^i}\left(\frac{da^{e_1}\pi^{-\epsilon_{1,i_0}}}{a^{e_1}\pi^{-\epsilon_{1,i_0}}} + \epsilon_{1,i_0}\frac{d\pi}{\pi}\right)\wedge \cdots \wedge \left(\frac{da^{e_n}\pi^{-\epsilon_{n,i_0}}}{a^{e_n}\pi^{-\epsilon_{n,i_0}}} + \epsilon_{n,i_0}\frac{d\pi}{\pi}\right)\\
&=&\frac{\omega_i}{\pi^i} + \frac{\nu_i}{\pi^i}\wedge\frac{d\pi}{\pi}
\end{eqnarray*}
with $\omega_i \in \Omega^{n}_{K_1}$ and $\nu_i \in \Omega^{n-1}_{K_1}$. $g_{\widehat K}$ is once again in a form in which we can apply Proposition \ref{p2};  $g_{\widehat K} = g_0$. Because the residue degree $f(v/v_{{\ell,n}})$ is prime to $p$, $K_1$ is a prime to $p$ extension of $\overline{k_{\ell,n}(z_1,\ldots,z_r)} \cong k_{\ell-1,n}(x_\ell,y_{\ell,1},\ldots, y_{\ell,n-1},z_1,\ldots,z_r)$, a purely transcendental extension of $k_{\ell-1,n}$. Let $w$ be an extension of the $y_{\ell-1,n}$-adic valuation on $k_{\ell-1,n}(x_\ell,y_{\ell,1},\ldots, y_{\ell,n-1},z_1,\ldots,z_r)$ to $K_1$ with prime to $p$ residue degree and ramification index and having completion $\widehat K_1$ and residue field $\overline K_1$. By Lemma \ref{lh2}(2b) there exists a differential basis $B' = \{a_1',\ldots, a_{\ell(n+1)+r-1}'\}$ for $K_1/k$ coming from $\overline K_1$ such that $K_1^p(a_1,\ldots,a_{\ell+n-3}) = K_1^p(a_1',\ldots,a_{\ell+n-3}')$ and $\partial_1(g_0)$ descends to $K_1^p(a_1',\ldots,a_{\ell+n-3}')$. Since $\partial_1(g_0) = \res{K_1}(\gen_k(n+1,\ell-1,p))$ and $K_1/k_{\ell-1,n}$ is a prime to $p$ extension of a purely transcendental extension of $k_{\ell-1,n}$, this contradicts our induction hypothesis.
\end{proof}

\begin{corollary}
$\ed_k(\gen_k(n+1,\ell,p);p) \geq \ell+n-1$.
\label{c21}
\end{corollary}
\begin{proof}
Let $K/k_{\ell,n}$ be a prime to $p$ extension and $k\subset E\subset K$ be a subfield with $\trdeg_k(E)=\ell+n-2$. Assume $\res K(\gen_k(n+1,\ell,p))$ descends to $E$. Fix an extension $v$ of the valuation $\vyln$ to $K$ with residue degree and ramification index prime to $p$. By Example \ref{ex1} and Proposition \ref{p1} there exists a differential basis $\{a_1,\ldots,a_{\ell(n+1)}\}$ of $K/k$ coming from $\overline K$ such that $E\subset K^p(a_1,\ldots, a_t)$ with $t \leq \ell+n-2$. Since $\res K(\gen_k(n+1,\ell,p))$ descends to $E$ it also descends to $K^p(a_1,\ldots,a_t)\subseteq K^p(a_1,\ldots, a_{\ell+n-2})$. This contradicts Theorem \ref{t22}.
\end{proof}

The next theorem improves the lower bound for the essential dimension of $\gen_k(n+1,\ell,p)$ given in Corollary \ref{c21} by one.

\begin{theorem}
For $\ell,n\geq 1$, 
\[\ed_k(\gen_k(n+1,\ell,p))\geq \ed_k(\gen_k(n+1,\ell,p);p)\geq \ell+n.\]
\label{t1}
\end{theorem}

\begin{proof}The proof is by induction on the symbol length $\ell$ and follows the outline of the proof of \cite[10.2]{bc}. For $\ell=1$ we are done by Lemma \ref{l3}. Fix $\ell>1$ and assume the theorem holds for all $\gen_k(n+1,\ell_0,p)$ for all algebraically closed fields $k$ of characteristic $p$, all $n\geq 1$ and all $\ell_0<\ell$. Let $K/k_{\ell,n}$ be a finite prime to $p$ extension. Assume there exists a field $k\subset E \subset k_{\ell,n}$ with $\trdeg_k(E)= \ell+n-1$ and $g \in \H^{n+1}_p(E)$ such that $\res{K}(g) = \res{K}(\gen_k(n+1,\ell,p))$. As usual fix $v_{{\ell,n}}$ to be the $y_{\ell,n}$-adic valuation on $k_{\ell,n}$ and fix $v$, an extension of $\vyln$ to $K$ with both $e(v/\vyln)$ and $f(v/\vyln)$ prime to $p$. Write $\widehat K$ and $\overline K$ for the completion and residue field of $K$. All first and second residues considered below will be with respect to the valuation $v$ on $\widehat K$.

\medskip
\noindent
{\it Case 1.} $w = v|_E$ is the trivial valuation. Let $g \in \H^{n+1}_p(E)$ be the class of the $n$-form
\[\sum b_i\frac{dc_{i,1}}{c_{i,1}}\wedge \cdots \wedge \frac{dc_{i,n}}{c_{i,n}} \in \Omega^{n}_E.\] 
Since $w(b_i)=w(c_{i,j}) = 0$ for all $i,j$, $\partial_2(\gen_k(n+1,\ell,p)_{\widehat K}) = \partial_2(g_{\widehat K})=0$. This contradicts Corollary \ref{cor56}.

\medskip
\noindent
{\it Case 2.} $w=v|_E$ is nontrivial and $p|e(v/w)$, the ramification index of $v$ over $w$. Since $K/\kln$ is a finite extension, $\trdeg_k(K) = \trdeg_k(k_{\ell,n})$ and $\trdeg_k(\overline K) = \trdeg_k(\overline k_{\ell,n})$. Hence $v$ is a geometric valuation on $K$ of rank 1. We can therefore apply Proposition \ref{p1} which says there exists a differential basis $\{a_1,\ldots, a_{\ell(n+1)}\}$ for $K/k$ coming from $\overline K$ so that $E\subset K^p(a_1,\ldots, a_t)$ with $t \leq \ell+n-1 = \trdeg_k(E)$. Moreover, since $p|e(v/w)$, $v(a_i)=0$ for all $1\leq i \leq t$ (see Remark \ref{r3}). Set $L=K^p(a_1,\ldots, a_t)$ and from this point forward we denote  $g = \res{L/E}(g) \in \H^{n+1}_p(L)$. By \cite[7.8]{E} (\cite[9.2]{bc}) the differential basis $\{\overline a_i\}$ for $\overline K/k$ corresponds to a coefficient field $K_1\subset \widehat K$ containing $\{a_i \,|\, v(a_i)=0\}$.  In particular, there is an isomorphism $\widehat K \cong K_1((\pi))$ with $\pi = a_{i_0}$, $i_0> t$ and $a_i \in K_1$ for $1\leq i \leq t$.  By Lemma \ref{lh2}, $g_{\widehat K} = \sum g_i$ with each $g_i$  a sum of elements of the form $f/\pi^i \frac{da^{e_1}}{a^{e_1}}\wedge \cdots \wedge \frac{da^{e_n}}{a^{e_n}}$ with $f \in K_1$ and $e_i \in \Lambda_t$. Since $a^{e_i}\in K_1$ for all $e_i \in \Lambda_t$, $f\frac{da^{e_1}}{a^{e_1}}\wedge \cdots \wedge \frac{da^{e_n}}{a^{e_n}} \in \Omega^n_{K_1}$ and therefore $g_{\widehat K} = \sum g_i$ satisfies the hypothesis of Proposition \ref{p2}. Since $g_{\widehat K} \in U_0(\widehat K)$, we can conclude $g_{\widehat K} = g_0$. Lemma \ref{lh2} part (1) shows $\partial_2(g_0) = 0$. This contradicts Corollary \ref{cor56}.

\medskip
\noindent 
{\it Case 3.} $w=v|_E$ is nontrivial and $p\nmid e(v/w)$. Set $e=e(v/w)$, let $\pi$ be a uniformizer for $(K,v)$ and $\tau = u\pi^e$ be a uniformizer for $(E,w)$ and $\mathcal B' = \{a_1,\ldots, a_{\ell+n-2},\tau\}$  a differential basis for $E/k$ coming from $\overline E$, the residue field of $E$ with respect to $w$. We want to show the subset $\{a_1,\ldots,a_{\ell+n-2}\}$ of $\mathcal B'$ extends to a differential basis for $K/k$. To do this we first prove that the set 
\begin{equation}\{a^{e} \,|\, e \in \Lambda_{\ell+n-2}\}\subset K\label{e989}\end{equation}
is linearly independent over $K^p$. Recall in the proof of Theorem \ref{p1} we  took a minimal generating set of these elements over $K^p$ to build a full differential basis of $K$.  In general the set in (\ref{e989}) won't be linearly independent over $K^p$, but it is in our case because we have assumed something special on $E$, that is, $\res{K}(\gen_k(n+1,\ell,p))$ descends to $E$ and we can harness the power of Theorem \ref{t22}. 

Assume that up to numbering there exists $t<\ell+n-2$ such that $a_1,\ldots, a_t$ are a minimal system of generators of $K^p(a_1,\ldots, a_{\ell+n-2})$ over $K^p$, i.e., $K^p(a_1,\ldots, a_{\ell+n-2},\tau) = K^p(a_1,\ldots, a_t,\tau)$. Arguing as in the proof of \ref{p1}, there exists a differential basis for $K/k$, $\{a_1',\ldots, a_{\ell(n+1)-1}',\tau\}$, such that $K^p(a_1,\ldots, a_t, \tau) = K^p(a_1', \ldots, a_t',\tau)$. Since $t+1\leq \ell+n-2$ and $\res K(\gen_k(n+1,\ell,p))$ descends to $E$, the inclusion 
\[E\subset  K^p(a_1,\ldots, a_t,\tau)= K^p(a_1', \ldots, a_t',\tau)\subset  K^p(a_1', \ldots, a_{\ell(n+1)-1}',\tau)=K\]
contradicts Theorem \ref{t22}. The set in (\ref{e989}) is therefore linearly independent over $K^p$ and we may choose $a_{\ell+n-1},\ldots,a_{\ell(n+1)-1} \in R^\times$ such that $\mathcal B = \{a_1,\ldots,a_{\ell(n+1)-1},\pi\}$ is a $p$-basis for $K/k$ and hence a differential basis for $K/k$ coming from $\overline K$. We have lined up our two completions $\widehat E \subset \widehat K$ to admit compatible coefficient fields. That is, we can choose coefficient fields $E_1$ and $K_1$ of $\widehat E$ and $\widehat K$ respectively so that
\[\widehat E \cong E_1((\tau)) \subset \widehat K \cong K_1((\pi))\] 
and $E_1 \subset K_1$. Since these coefficient fields correspond to the units in the differential bases $\mathcal B'$ and $\mathcal B$ respectively, we have $\{a_1,\ldots, a_{\ell+n-2}\}$ is a differential basis for $E_1/k$ and $\{a_1,\ldots, a_{\ell(n+1)-1}\}$ is a differential basis for $K_1/k$. Note in particular that the transcendence degree of $E_1$ over $k$ is $\ell+n-2$, the order of the differential basis (\cite[16.14]{E}).

We now show $g \in \H^{n+1}_{p,ur}(\widehat E) = U_0(\widehat E)$. Let $m$ be the smallest integer such that $g \in U_m(\widehat E)$. Assume $m > 0$ and consider the map 
 \begin{equation}U_m/U_{m-1}(\widehat E) \to U_{em}/U_{em-1}(\widehat K)\label{eq6}\end{equation}
 from Lemma \ref{l4}. If $p \nmid m$, then $p \nmid em$ and by Lemma \ref{l4}(\ref{e66}) the differential form side of the commutative diagram is multiplication by $\bar u^{-m}$ (recall $\tau = u\pi^e$ with $v(u)=0$) composed with extension of scalars.  In particular, (\ref{eq6}) is an injection if and only if $\Omega^n_{E_1} \to \Omega^n_{K_1}$ is an injection. By Lemma \ref{l78}, the set inclusion $\{a_1,\ldots, a_{\ell+n-2}\} \subseteq \{a_1,\ldots, a_{\ell(n+1)-1}\}$ of $p$-bases for $E_1$ and $K_1$ shows that $\Omega^{n}_{E_1} \to \Omega^{n}_{K_1}$ is an injection. Since $g_{\widehat K} \in U_0(\widehat K)\subset U_{em-1}(\widehat K)$, this is a contradiction to the minimality of $m$.

Assume $m>0$ and $p | m$. Let $\rho_m^{-1}(g_{\widehat E}) = (\omega_{n} + \Omega^{n}_{E_1,d=0}, \omega_{n-1} + \Omega^{n-1}_{E_1,d=0})$ with $\omega_i \in \Omega^i_{E_1}$. Since $g_{\widehat K} \in U_0(\widehat K)\subset U_{em-1}(\widehat K)$ Lemma \ref{l4}(\ref{ee66}) gives
 \begin{eqnarray}
(0,0) &=&\psi_m\left(\omega_{n} + \Omega^{n}_{E_1,d=0}, \omega_{n-1} + \Omega^{n-1}_{E_1,d=0}\right)\label{e65} \\
&=& \left(\bar u^{-m}\omega_{n}+\omega_{n-1}\wedge\frac{d\bar u}{\bar u} + \Omega^{n}_{K_1,d=0},e\bar u^{-m}\omega_{n-1} + \Omega^{n-1}_{K_1,d=0}\right) \nonumber
\end{eqnarray}
Since $e\bar u^{-m} \in K_1^p-\{0\}$, (\ref{e65}) shows $d(\omega_{n-1}) =0$ in $\Omega^{n}_{K_1}$. Therefore, by Lemma \ref{l78}, $d(\omega_{n-1})=0$ in $\Omega^{n}_{E_1}$ and also $d( \omega_{n-1}\wedge d\bar u/\bar u) = 0$, i.e., 
\[\overline u^{-m}\omega_{n}+\omega_{n-1}\wedge \frac{d\bar u}{\bar u} + \Omega^{n}_{K_1,d=0} = \overline u^{-m}\omega_{n}+ \Omega^{n}_{K_1,d=0}\]
In particular, $\omega_{n} + \Omega^{n}_{E_1,d=0} \in \ker(\Omega^{n}_{E_1}/\Omega^{n}_{E_1,d=0} \to \Omega^{n}_{K_1}/\Omega^{n}_{K_1,d=0})$ which is $0$ by Lemma \ref{l78}. Therefore, $\omega_{n} \in \Omega^{n}_{E_1,d=0}$. We have shown that $g_{\widehat E} \in \ker(\rho_m^{-1}) =0$, i.e., $g_{\widehat E} \in U_{m-1}(\widehat E)$, contradicting the minimality of $m$.  Therefore, $g \in U_0(\widehat E)$ and we can use Lemma \ref{l4}(\ref{e45}); set $\rho_0^{-1}(g_{\widehat E}) = (\omega_{n},\omega_{n-1})$ where $\omega_i$ is a $i$-form representing a class in $\H^{i+1}_p$. 

\begin{equation}\xymatrix{
\H^{n+1}_\ur(\widehat K) \ar[r]^(.4){\rho_0^{-1}} & \H^{n+1}_{p}(K_1) \oplus \H^{n}_p(K_1)&(\omega_{n} + \omega_{n-1}\wedge\frac{d\bar u}{\bar u}, e\,\omega_{n-1})\\
\H^{n+1}_\ur(\widehat E) \ar[u]^{\mathrm{res}}\ar[r]^(.4){\rho_0^{-1}}&\H^{n+1}_{p}(E_1) \oplus \H^{n}_p(E_1)\ar[u]_{\psi_0}&(\omega_{n},\omega_{n-1})\ar@{|->}[u]_{\psi_0}
}
\label{eq88}
\end{equation}
Let $\gen_k(n+1,\ell,p)_{\widehat K}$ denote the extension of scalars from $k_{\ell,n}$ to $\widehat K$ and set $y_{\ell,n} = u'\pi^{e'}$ with $u'$ a unit in $K$ and $e' = e(v/v_{\ell,n})$, an integer prime to $p$. By Lemma 4.3(3) the map $\rho_0^{-1} = (\partial_1,\partial_2)$ applied to $\gen_k(n+1,\ell,p)_{\widehat K}$ is:
\begin{eqnarray*}
\partial_1(\gen_k(n+1,\ell,p)_{\widehat K})&=& \gen_k(n+1,\ell-1,p)_{K_1}+x_\ell \frac{dy_{\ell,1}}{y_{\ell,1}}\wedge\cdots\wedge \frac{dy_{\ell,n-1}}{y_{\ell,n-1}}\wedge \frac{d\overline{u'}}{\overline{u'}}\\
\partial_2(\gen_k(n+1,\ell,p)_{\widehat K})&=&e'\, x_\ell \frac{dy_{\ell,1}}{y_{\ell,1}}\wedge\cdots\wedge \frac{dy_{\ell,n-1}}{y_{\ell,n-1}}
\end{eqnarray*}
Combining this with (\ref{eq88}) we have two equalities in $\H^{n+1}_p(K_1)$ and $\H^n_p(K_1)$ respectively:
\begin{eqnarray}
\gen_k(n+1,\ell-1,p)_{K_1} + x_\ell \frac{dy_{\ell,1}}{y_{\ell,1}}\wedge\cdots\wedge \frac{dy_{\ell,n-1}}{y_{\ell,n-1}}\wedge \frac{d\overline{u'}}{\overline{u'}}&=& \omega_{n}+\omega_{n-1}\wedge\frac{d\bar u}{\bar u}\label{e678}\\
e'\, x_\ell \frac{dy_{\ell,1}}{y_{\ell,1}}\wedge\cdots\wedge \frac{dy_{\ell,n-1}}{y_{\ell,n-1}}&=& e\,\omega_{n-1}\nonumber
\end{eqnarray}

If $\omega_{n-1}\wedge \frac{d\bar u}{\bar u}$ and $x_\ell \frac{dy_{\ell,1}}{y_{\ell,1}}\wedge\cdots\wedge \frac{dy_{\ell,n-1}}{y_{\ell,n-1}}\wedge \frac{d\overline{u'}}{\overline{u'}}$ were 0, then $\gen_k(n+1,\ell-1,p)_{K_1}$ would descend to $E_1$, a field with transcendence degree $n+\ell-2 = n+(\ell-1)-1$ over $k$ (\cite[VI.10.3, Cor 4]{MR979760}) and we would proceed by analyzing and manipulating the extensions $E_1$ and $K_1$ to contradict the induction hypothesis. But these are not necessarily zero, so to get our contradiction we need to split them.

Let $k'$ be an algebraic closure of $k(x_\ell,y_{\ell,1},\ldots,y_{\ell,n-1})$ and set 
\[k'_{\ell-1,n}=k'(x_i,y_{i,j})_{1\leq i \leq \ell-1, \,1\leq j \leq n},\] 
so that $\gen_{k'}(n+1,\ell-1,p) \in \H^{n+1}_p(k'_{\ell-1,n})$. We will derive a contradiction to the induction hypothesis on this length $\ell-1$ generic $p$-symbol. Let $K_1'$ be the composite of $K_1$ and $k'_{\ell-1,n}$ (both fields are contained in an algebraic closure of $k'_{\ell,n}$) and note that $K_1'/k'_{\ell-1,n}$ is of degree prime to $p$. Set $E_1'$ to be the composite of $E_1$ and $k'$ over $k$ (each of these fields are contained in $K_1'$). 
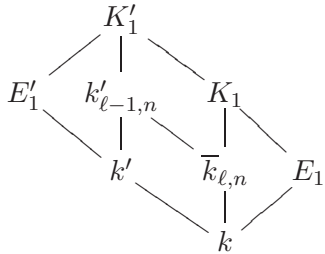
\begin{equation}
\xymatrix@1@C=.2in@R=.2in{
     &K_1'\ar@{-}[dr]\\
E_1'\ar@{-}[ur]&\kpln\ar@{-}[u]&K_1\\
                    & k'\ar@{-}[ul]\ar@{-}[u] &\okln\ar@{-}[u]\ar@{-}[ul]&E_1\ar@{-}[ul]\\
                    &&k\ar@{-}[ur]\ar@{-}[u]\ar@{-}[ul]
}
\end{equation}
We now extend scalars: $\H^{n+1}_p(K_1)\to\H^{n+1}_p(K_1')$. First note that $\gen_k(n+1,\ell-1,p)_{k'_{\ell-1,n}} = \gen_{k'}(n+1,\ell-1,p)$. Also note that since the $y_{\ell,i}$ are $p$-th powers in $k'$, 
\[\res{k'_{\ell-1,n}}\left(x_\ell\frac{dy_{\ell,1}}{y_{\ell,1}}\wedge \cdots\wedge\frac{dy_{\ell,n-1}}{y_{\ell,n-1}}\right)=0.\]
In particular, $\res{k'_{\ell-1,n}/E_1}(e\,\omega_{n-1}) = 0$ and since $p\nmid e$, $\res{K_1'}(\omega_{n-1}) = 0$. Therefore, extending scalars all the way up to $\H^{n+1}_p(K_1')$, the two equations in (\ref{e678}) collapse to
\[\res{K_1'/k'_{\ell-1,n}}(\gen_{k'}(n+1,\ell-1,p)) = \res{K_1'/E_1'}(\res{E_1'/E_1}(\omega_n)).\] Since $K'_1/k'_{\ell-1,n}$ is a prime to $p$ extension and the field $E_1'$ satisfies $\trdeg_{k'}(E_1')\leq \trdeg_{k}(E_1)=n+(\ell-1)-1$, this contradicts the induction hypothesis.

\end{proof}

Let $n=2$, then $\gen(2,\ell,p) = \sum_{i=1,\ell}x_i\frac{dy_i}{y_i} \in \H^2_p(K_{\ell,1})$ is the class of the generic length $\ell$ $p$-symbol division algebra $D_\ell = \otimes_{i=1}^\ell[x_i,y_i)$ in ${}_p\br(k(x_1,y_1,\ldots,x_\ell,y_\ell))$. Combining Theorem \ref{t1} with \cite[3.2]{baek}, we get the $p$-essential dimension of $D_\ell$ as a $p$-torsion Brauer class:

\begin{corollary}
$\ed_k(D_\ell;p) =\ed_k(D_\ell) =  \ell+1$. 
\label{c34}
\end{corollary}

Fix an algebraically closed field $k$ of characteristic $p$. Recall 
$\alg_{p^\ell,p^r}:\mathbf{Fields}/k \to \mathbf{sets}$ is the functor taking a field extension $K/k$ to the set  of isomorphism classes of central simple algebras over $K$ of degree $p^\ell$ and exponent dividing $p^r$. As mentioned in the introduction there is a natural bijection between $\H^1(K,\mathrm{GL}_{p^\ell}/\mu_{p^r})$ and $\alg_{p^\ell,p^r}(K)$ (see \cite[Ex. 1.1]{MR2520892}). In particular, $\ed_k(\alg_{p^\ell,p^r}) = \ed_k(\mathrm{GL}_{p^\ell}/\mu_{p^r})$ and $\ed_k(\alg_{p^\ell,p^r};p) = \ed_k(\mathrm{GL}_{p^\ell}/\mu_{p^r};p)$.

\begin{corollary} $\ed_k(\mathrm{GL}_{p^\ell}/\mu_p;p) = \ed_k(\alg_{p^\ell,p};p)\geq \ell+1$. 
\label{main}
\end{corollary}

\begin{proof} By Corollary \ref{c34} $D_\ell$ is an algebra defined over an extension of $k$ with degree $p^\ell$, exponent $p$ and essential dimension $\ell+1$ as a $p$-torsion Brauer class. The $p$-essential dimension of $D_\ell$ as an element of $\alg_{p^\ell,p}(K)$ is at least $\ell+1$.
\end{proof}

The author would like to thank Skip Garibaldi for introducing the problem.  Thank you also to Parimala and Suresh Venapally at Emory University, for their helpful comments, time and support during a visit. The author would also like to thank Stephan Tillmann (ARC Discovery Grant DP140100158) at the University of Sydney for support during the writing of this paper.

\bibliographystyle{abbrv}
\bibliography{ED}

\end{document}